\documentclass[aap,MSNbibl,seceqn,citesort,dvips]{arximspdf}
\usepackage{amscd}
\usepackage{mathbh}

\doi{10.1214/10-AAP738}
\volume{22}
\issue{2}
\pubyear{2012}
\firstpage{477}
\lastpage{521}

\makeatletter

\newtheorem{theorem}{Theorem}[section]
\newproclaim{remark}[theorem]{Remark}
\newproclaim{example}[theorem]{Example}
\newtheorem{lemma}[theorem]{Lemma}
\newproclaim{defn}[theorem]{Definition}
\newproclaim{ass}{Assumption}[section]
\newtheorem{prop}[theorem]{Proposition}

\newcommand{\R}{{\mathbb R}}
\newcommand{\N}{{\mathbb N}}
\newcommand{\Z}{{\mathbb Z}}
\newcommand{\indd}{{\mathbh1}}
\newcommand{\E}{\mathbb{E}}
\newcommand{\PP}{\mathbb{P}}
\newcommand{\cal}{\mathcal}
\newcommand{\supp}{\operatorname{supp}}

\newcommand{\patreq}{r}
\newcommand{\ME}{\mathcal{E}^{(N)}_0}
\newcommand{\calQ}{{\cal Q}}

\newcommand{\lan}{\langle}
\newcommand{\ran}{\rangle}

\newcommand{\ve}{\varepsilon}
\newcommand{\ra}{\rightarrow}
\newcommand{\wconv}{\stackrel{w}{\rightarrow}}
\newcommand{\swconv}{\Rightarrow}
\newcommand{\eem}{{\cal M}(E)}
\newcommand{\eemf}{{\cal M}_F(E)}
\newcommand{\mm}{{\cal M}[0,H)}
\newcommand{\mmf}{{\cal M}_F[0,H)}
\newcommand{\mmfs}{{\cal M}_F[0,H^s)}
\newcommand{\mmds}{{\cal M}_D[0,H^s)}
\newcommand{\mmfr}{{\cal M}_F[0,H^r)}
\newcommand{\mmdr}{{\cal M}_D[0,H^r)}
\newcommand{\dspaceh}{{\cal D}_{\cal H}[0,\infty)}
\newcommand{\dspacee}{{\cal D}_{E}[0,\infty)}
\newcommand{\dTspacee}{{\cal D}_{E}[0,T]}
\newcommand{\incspace}{{\cal I}_{\R_+}[0,\infty)}
\newcommand{\ecb}{{\cal C}_b(E)}
\newcommand{\ecc}{{\cal C}_c (E)}
\newcommand{\cb}{{\cal C}_b(\R_+)}
\newcommand{\newf}{\varphi}
\newcommand{\newft}{\varphi(\cdot,t)}
\newcommand{\newfs}{\varphi(\cdot,s)}
\newcommand{\dxnewfs}{\newf_x(\cdot,s)}
\newcommand{\dtnewfs}{\newf_s(\cdot,s)}
\newcommand{\zerof}{{\mathbf0}}
\newcommand{\invrenegs}{\eta_*}
\newcommand{\fmeasns}{\overline{\nu}{}^{(N)}_{*}}
\newcommand{\measns}{\nu^{(N)}_{*}}
\newcommand{\fyns}{\overline{Y}{}^{(N)}_{*}}
\newcommand{\frenegns}{\overline{\eta}{}^{(N)}_{*}}
\newcommand{\renegns}{\eta^{(N)}_{*}}
\newcommand{\yn}{Y^{(N)}}
\newcommand{\fyn}{\overline{Y}{}^{(N)}}
\newcommand{\fxns}{\overline{X}{}^{(N)}_{*}}
\newcommand{\xns}{X^{(N)}_{*}}
\newcommand{\fren}{\overline{\alpha}{}^{(N)}_E}
\newcommand{\ren}{\alpha_E^{(N)}}
\newcommand{\re}{\alpha_E}
\newcommand{\idlen}{I^{(N)}}
\newcommand{\sn}{S^{(N)}}
\newcommand{\qn}{Q^{(N)}}
\newcommand{\dn}{D^{(N)}}
\newcommand{\rn}{R^{(N)}}
\newcommand{\kn}{K^{(N)}}
\newcommand{\fkn}{\overline{K}{}^{(N)}}
\newcommand{\en}{E^{(N)}}
\newcommand{\agen}{a^{(N)}}
\newcommand{\waitn}{w^{(N)}}
\newcommand{\measn}{\nu^{(N)}}
\newcommand{\meas}{\nu}
\newcommand{\renegn}{\eta^{(N)}}
\newcommand{\reneg}{\eta}
\newcommand{\fqn}{\overline{Q}{}^{(N)}}
\newcommand{\fq}{\overline{Q}}
\newcommand{\fe}{\overline{E}}
\newcommand{\fen}{\overline{E}{}^{(N)}}
\newcommand{\fmeasn}{\overline{\nu}{}^{(N)}}
\newcommand{\frenegn}{\overline{\eta}{}^{(N)}}
\newcommand{\fmeas}{\overline{\nu}}
\newcommand{\freneg}{\overline{\eta}}
\newcommand{\flam}{\overline{\lambda}}
\newcommand{\frn}{\overline{R}{}^{(N)}}
\newcommand{\fr}{\overline{R}}
\newcommand{\fntf}{F^{\overline\eta_t}}
\newcommand{\fidlen}{\overline{I}{}^{(N)}}
\newcommand{\xn}{X^{(N)}}
\newcommand{\fxn}{\overline{X}{}^{(N)}}
\newcommand{\fx}{\overline{X}}
\newcommand{\newfk}{\overline{Z}}
\newcommand{\newspace}{{\cal S}_0}
\newcommand{\fk}{\overline{K}}
\newcommand{\invmeas}{\nu_*}

\makeatother

\begin{document}
\begin{frontmatter}

\title{Asymptotic approximations for
stationary distributions of many-server queues with~abandonment}
\runtitle{Stationary distributions of many-server queues}

\begin{aug}
\author[A]{\fnms{Weining} \snm{Kang}\ead[label=e1]{wkang@umbc.edu}} and
\author[B]{\fnms{Kavita} \snm{Ramanan}\thanksref{t1}\ead
[label=e2]{kavita@dam.brown.edu}}
\runauthor{W. N. Kang and K. Ramanan}
\affiliation{University of Maryland, Baltimore County and Brown University}
\address[A]{Department of Mathematics and Statistics\\
University of Maryland, Baltimore County\\
1000 Hilltop Circle\\
Baltimore, Maryland 21250\\
USA\\
\printead{e1}}
\address[B]{Division of Applied Mathematics \\
Brown University \\
Providence, Rhode Island 02912\\
USA\\
\printead{e2}}
\end{aug}

\thankstext{t1}{Supported in part
by NSF Grants CMMI-0728064, CMMI-0928154.}

\received{\smonth{3} \syear{2010}}
\revised{\smonth{8} \syear{2010}}

%
\begin{abstract}
A many-server queueing system is considered in which customers arrive
according to a renewal process and have service and patience times that
are drawn from two independent sequences of independent, identically
distributed random variables. Customers enter service in the order of
arrival and are assumed to abandon the queue if the waiting time in
queue exceeds the patience time. The state of the system with $N$
servers is represented by a four-component process that consists of the
forward recurrence time of the arrival process, a~pair of
measure-valued processes, one that keeps track of the waiting times of
customers in queue and the other that keeps track of the amounts of
time customers present in the system have been in service and a
real-valued process that represents the total number of customers in
the system. Under general assumptions, it is shown that the state
process is a Feller process, admits a stationary distribution and is
ergodic. It is also shown that the associated sequence of scaled
stationary distributions is tight, and that any subsequence converges
to an invariant state for the fluid limit. In particular, this implies
that when the associated fluid limit has a unique invariant state, then
the sequence of stationary distributions converges, as $N \ra\infty$,
to the invariant state. In addition, a simple example is given to illustrate
that, both in the presence and absence of abandonments, the $N \ra
\infty$ and $t \ra\infty$ limits cannot always be interchanged.
\end{abstract}

%

\end{frontmatter}

\section{Introduction}
\label{sec-intro}

\subsection{Description}
\label{subs-back}

An $N$-server queueing system is considered in which customers arrive
according to a renewal process, have independent and identically distributed
(i.i.d.) service requirements that are drawn from a general
distribution with
finite mean and also carry i.i.d. patience times that are drawn from
another general distribution.
Customers enter service in the order of arrival as soon as an
idle server is available, service is nonpreemptive, and customers
abandon the queue if the time spent waiting in queue reaches the
patience time. This system is also sometimes referred to as the $GI/GI/N+G$
model. In this work, it is assumed
that the sequences of service requirements and
patience times are mutually independent, and that
the interarrival, service and patience time distributions
have densities.

The state of the $N$-server system is represented by a four component process
$Y^{(N)}$, consisting of the forward recurrence time process associated with
the renewal arrival process, a measure-valued process that keeps track
of the
amounts of time customers currently in service have been in service, another
measure-valued process that encodes the times elapsed since customers have
entered the system (for all customers for which this time has not yet exceeded
their patience times) and a real-valued process that keeps track of the
total number of customers in the system. This infinite-dimensional state
representation was shown in Lemma B.1 of
Kang and Ramanan \cite{kanram08b} to lead to a Markovian description
of the dynamics (with respect to a suitable filtration).
In addition, a fluid limit for this model was also
established in~\cite{kanram08b}, that is, under suitable assumptions,
it was shown that almost surely, $\overline{Y}{}^{(N)}= Y^{(N)}/N$ converges,
as $N \ra\infty$,
to a limit process $\overline{Y}$ which is characterized as the
unique solution to a set of coupled integral equations
(see Definition \ref{def-fleqns}). The process $\overline{Y}$ will be
referred to as the fluid limit.

The present work focuses on obtaining first-order
approximations to the stationary distribution
of $Y^{(N)}$ which is of fundamental interest for the performance
analysis of
many-server queues.
It is first shown that for each $N$,
$Y^{(N)}$ is a Feller, strong Markov process and has a stationary
distribution. Under an additional assumption (Assumption
\ref{assPosF}), uniqueness of the stationary distribution and
ergodicity of each $Y^{(N)}$ is also established.
The main result, Theorem \ref{thm-convstat},
shows that under fairly general assumptions the
sequence of stationary distributions is tight and that
any subsequential limit is an invariant state for the fluid limit.
In particular,
if the fluid limit has a unique invariant state,
this implies that the sequence of
scaled stationary distributions
(indexed by the number of servers $N$) converges, as $N \ra\infty$,
to this unique invariant state.
More generally, this work seeks to illustrate how an infinite-dimensional
Markovian representation of a stochastic network can facilitate the
(first-order) characterization of the associated stationary distributions.
Furthermore, examples are presented
to illustrate several subtleties in the dynamics.
Specifically, it is shown that the presence of a unique invariant
state is not a necessary condition for
the sequence of scaled stationary distributions to have a limit and
that even when such a limit exists,
the $t \ra\infty$ and $N \ra\infty$ limits cannot in general
be interchanged.

\subsection{Motivation and context}

The study of many-server queueing systems with abandonment is motivated by
applications to telephone call centers and (more generally) customer
contact centers.
The incorporation of customer abandonment captures the effect of customers'
impatience, which has
a substantial impact on the performance of the system. For example, customer
abandonment can stabilize a system even when it is overloaded.
A~considerable body of work has been
devoted to the study of various steady-state or stationary performance
measures of many-server
queues, both with and without abandonment.
In the absence of abandonment,
when the interarrival times and service times are exponential, an explicit
expression for the steady state queue length can be found in Bocharov
et al.
\cite{BDPS}. In the discrete-time setting, when the
i.i.d. interarrival and service times are generally distributed,
the classical work of Kiefer and Wolfowitz \cite{KiWo}
(see also Foss \cite{Foss}) establishes the convergence in distribution,
as time goes to infinity, of the waiting time vectors to a stationary
limit.
The generalization to continuous time is dealt with in Asmussen and Foss
\cite{asmfoss93}.
For a many-server queue with stationary renewal arrivals, deterministic
service times and no abandonments, Jelenkovic, Mandelbaum and Mom\v
{c}ilovi\'{c} \cite{jelmanmom}
showed that on the diffusive scale, the scaled stationary waiting
times converge in distribution to the supremum of a Gaussian random
walk with
negative drift.
For a many-server queue with stationary renewal arrivals, a finitely
supported, lattice-valued service time distribution and no abandonments,
in the so-called Halfin--Whitt asymptotic regime where the number of servers
$N$ goes to infinity and the corresponding arrival rate grows as
$N-\beta\sqrt{N}$
for some $\beta>0$,
Gamarnik and
Mom\v{c}ilovi\'{c} \cite{GaMo08} characterized the limit
of the scaled stationary queue length distribution in terms of the stationary
distribution of an explicitly constructed Markov chain and obtained an
explicit expression for the exponential decay rate of the moment generating
function of this limiting stationary distribution.

For many-server queues with abandonment whose interarrival,
service and abandonment distributions are exponential,
Garnett, Mandelbaum and Reiman
\cite{garmanrei02} provide exact calculations of various steady state
performance measures and their approximations in the Halfin--Whitt asymptotic
regime,
both in the case of finite waiting rooms ($M/M/N/B+M$) and infinite
waiting rooms ($M/M/N+M$).
In the case of Poisson arrivals, exponential service distribution and general
abandonment distribution ($M/M/N+G$), explicit formulae for the steady state
distributions of the queue length and virtual waiting time were
obtained by
Baccelli and Hebuterne \cite{bachet81} (see Sections IV and V.2 therein),
whereas several other steady state performance measures and their
approximations in the Halfin--Whitt asymptotic regime were derived by
Mandelbaum and Zeltyn \cite{manzel05}.

In the previously mentioned works on characterization of
stationary distributions of many-server queues, either
the interarrival times and service times are assumed to be
exponential
or it is assumed that the service time distribution is discrete
and has a finite support, and
that there is no abandonment.
However, statistical
analysis of real call centers has shown that both service times and patience
times are typically not exponentially distributed (see Brown et al.
\cite{brownetal} and Mandelbaum and Zeltyn \cite{manzel05}).
In general, it is difficult to derive explicit expressions for the
stationary distributions of many-server queues,
especially in the more realistic case when service times are not exponential
and there is abandonment.
This is also the case for many other classes of stochastic
networks.
To circumvent this problem, a~common approach that is
taken is to identify the long-time limits of the fluid or diffusion
approximations, which are often more tractable, and then use these
limits as
approximations of the stationary distribution of the original system.
Such an approach relies on the premise that the
long-time behavior of the fluid limit can be characterized
and also requires an argument that justifies the
interchange of (the $N \ra\infty$ and $t \ra\infty$)
limits (see, e.g., Gamarnik and Zeevi \cite{gamzee} for an interchange
of limits result in the context of
generalized Jackson networks).
However, we show that this approach may not always be appropriate
for stochastic network models. Indeed, for the case of many-server
queues whose service distributions are not
exponential,
the long-time behavior of the fluid is subtle and difficult
to characterize
in large part due to the complexity in the
dynamics introduced by the coupling of the measure-valued component
of the fluid limit
with the positive real-valued component by the nonidling condition.
Furthermore, as the example we construct
in Section \ref{subs-counteregs} demonstrates,
in general, the order of the $N \ra\infty$ and $t \ra\infty$
limits cannot be interchanged.



Instead we take a different approach to showing
convergence that is more appropriate for mean-field
limits, which involves establishing tightness of the
stationary distributions and showing that any subsequence converges to
an invariant state.
A~more detailed description of the approach is provided in Section
\ref{subs-mainresults} and additional discussion is provided
in Section \ref{subs-counteregs}.
The present work is also related to the work of Whitt \cite{whifluid06} who
analyzed a discrete time version of the model,
proposed a fluid limit model
and made several conjectures on the associated steady-state quantities.
A comparison of our results with those of Whitt \cite{whifluid06}
is also given in Section \ref{subs-mainresults} after the
statement of our main results.


\subsection{Outline}

The outline of the paper is as follows. A precise mathematical
description of
the model is provided in Section \ref{sec-mode}.
Section \ref{subs-mainres} introduces the basic assumptions and states the
main result. The Feller property and the existence of stationary distributions
of the state descriptor are proved in Section \ref{secSD}.
The fluid equations and the invariant manifold are described in Section
\ref{sec-Fluid} and the asymptotics of the stationary distributions is
established in Section \ref{secconv}. Finally, Section \ref{subs-counteregs}
contains a discussion of the positive Harris recurrence and ergodicity
of the
state descriptor, the long time behavior of the fluid limit and
an example that shows that
the ``interchange of limits'' property does not always hold. In the
remainder of this section, we introduce some common notation used in
the paper.

\subsection{Notation and terminology}
The following notation will be used\break throughout the paper.
$\Z$ is the set of integers, $\N$ is the set of positive integers,
$\R$ is the set of real numbers, $\Z_+$ is the set of nonnegative
integers and
$\R_+$ the set of nonnegative real numbers.
For $a, b \in\R$, $a \vee b$ denotes the maximum of $a$ and $b$,
$a \wedge b$ the minimum of $a$ and $b$ and the short-hand $a^+$ is
used for $a \vee0$.
$\indd_B$ denotes the indicator function of the set $B$
[i.e., $\indd_B (x) = 1$ if $x \in B$ and $\indd_B(x) = 0$ otherwise].

\subsubsection{Function and measure spaces}
\label{subsub-funmeas}

Given any metric space $E$, $\ecb$ and $\ecc$ are, respectively,
the space of bounded, continuous functions and
the space of continuous real-valued functions with compact support
defined on
$E$.
The support of a function $\newf$ is denoted by $\supp(\newf)$.
We denote by $\dTspacee$
(resp., $\dspacee$) the space of $E$-valued,
c\`{a}dl\`{a}g functions on $[0,T]$ (resp., $[0,\infty)$)
and we endow this space with the usual Skorokhod $J_1$-topology
\cite{parbook}. When $E$ is Polish then
$\dTspacee$ and $\dspacee$ are also Polish
spaces (see \cite{parbook}).
Let $\incspace$ be the subset of nondecreasing functions
$f \in{\cal D}_{\R_+}[0,\infty)$ with $f(0) = 0$. Given $f\in
\incspace$,
$f^{-1}$ denotes the inverse function of $f$ defined by
%
%
\begin{equation}
\label{def-finv} f^{-1}(y)=\inf\{x\geq0\dvtx f(x)\geq y\}.
\end{equation}

The space of Radon measures on a complete separable metric space $E$,
endowed with the Borel $\sigma$-algebra,
is denoted by $\eem$, while
$\eemf$ is the subspace of finite measures in $\eem$.
Recall that a Radon measure on~$E$ is one that assigns finite measure
to every relatively
compact subset of $E$.
The space $\eemf$ is equipped with the weak topology, that is, a
sequence of measures
$\{\mu_n\}$ in $\eemf$ is said to converge to $\mu$ in the weak
topology (denoted
$\mu_n \wconv\mu$) if and only if for every $\newf\in\ecb$,
%
%
\begin{equation}
\label{w-limit}
\int_{E} \newf(x) \mu_n(dx) \ra\int_E \newf(x) \mu(dx)
\qquad\mbox{as } n \ra\infty.
\end{equation}
As is
well known, $\eemf$, endowed with the weak
topology is a Polish space.
The symbol $\delta_x$ will be used to denote the measure with unit
mass at the point $x$ and,
by some abuse of notation, we will use $\zerof$ to denote the
identically zero
Radon measure on $E$. When $E$ is an interval, say $[0,H)$ for some $H
\in(0,\infty]$,
we will often write ${\cal M}[0,H)$ and ${\cal M}_F[0,H)$
instead of ${\cal M}([0,H))$ and ${\cal M}_F([0,H))$, respectively.
For any $\mu\in{\cal M}_F[0,H)$, we define
%
%
\begin{equation}
\label{def-fmu}
F^\mu(x) \doteq\mu[0,x],\qquad x \in[0,H).
\end{equation}
For any
Borel measurable function $f\dvtx[0,H) \ra\R$ that is integrable
with respect to $\xi\in\mm$, we often use the short-hand notation
\[
\lan f, \xi\ran\doteq\int_{[0,H)} f(x) \xi(dx).
\]
Also, for ease of notation, given $\xi\in\mm$ and an interval $(a,b)
\subset[0,M)$, we will
use $\xi(a,b)$ to denote $\xi((a,b))$.

\subsubsection{Measure-valued stochastic processes}

In this work, we will be interested
in c\`{a}dl\`{a}g ${\cal H}$-valued stochastic processes where ${\cal
H} = \mmf$
for some $H\le\infty$.
These are random elements that are defined on a probability space
$(\Omega, {\cal F}, \PP)$ and take values in $\dspaceh$,
equipped with the Borel $\sigma$-algebra (generated by open sets under the
Skorokhod $J_1$-topology).
A~sequence $\{X_n\}_{n \in\N}$ of c\`{a}dl\`{a}g, ${\cal H}$-valued
processes, with $X_n$ defined on the probability space $(\Omega_n,
{\cal F}_n, \PP_n)$,
is said to converge in distribution
to a c\`{a}dl\`{a}g ${\cal H}$-valued process $X$ defined on $(\Omega,
{\cal F}, \PP)$ if, for every bounded, continuous functional
$\Xi\dvtx\dspaceh\ra\R$, we have
\[
\lim_{n \ra\infty} \E_n[ \Xi(X_n) ] = \E[ \Xi
(X)],
\]
where $\E_n$ and $\E$ are the expectation operators with respect to
the probability measures
$\PP_n$ and $\PP$, respectively.
Convergence in distribution of $X_n$ to $X$ will be denoted by $X_n
\swconv X$.

\section{Description of model and state dynamics}
\label{sec-mode}

In Section \ref{subs-modyn} we describe the basic model, which is sometimes
referred to as the $GI/GI/N+G$ model. In Section \ref{secrepdyn} we introduce
the state descriptor and some auxiliary processes and also describe the state
dynamics. In Section \ref{subs-dyneq} we obtain a convenient representation
formula for expectations of linear functionals of the measure-valued
components of
the state process.
In Section \ref{subsub-filt} we introduce a
filtration with respect to which the state descriptor is an adapted, strong
Markov process.
This model was also considered in \cite{kanram08b}, where a functional
strong law of large numbers limit for the state descriptor was
established as the number of servers and the mean arrival rate both
tend to infinity.

\subsection{Model description and primitive data}
\label{subs-modyn}

Consider a queueing system with $N$ identical servers in which arriving
customers are served in a nonidling, first-come-first-serve (FCFS)
manner, that is, a newly
arriving customer immediately enters service if there are any idle
servers or,
if all servers are busy, then the customer joins the back of
the queue, and the customer at the head of the queue (if one is
present) enters
service as soon as a server becomes free.

It is assumed that customers are impatient and that a customer reneges
from the queue as soon
as the amount of time he or she has waited in the queue reaches his or
her patience time.
Service is nonpreemptive and
customers do not renege once they have entered service.
The patience times of customers are given by an i.i.d. sequence,
$\{\patreq_i, i\in\Z\}$, with common cumulative distribution
function $G^r$ on $[0,\infty]$,
while the service requirements of customers are given by another
i.i.d.
sequence, $\{v_i, i\in\Z\}$, with common cumulative distribution function
$G^s$ on $[0,\infty)$. For \mbox{$i\in\N$}, $\patreq_i$ and $v_i$,
respectively, represent the
patience time and the service requirement of the $i$th customer to
enter the
system after time zero, whereas $\{\patreq_i, i\in-\N\cup\{0\}\}$
and $\{v_i, i\in-\N\cup\{0\}\}$, respectively, represent the
patience times and the
service requirements of customers that arrived prior to time zero (if such
customers exist), ordered according to their arrival times
(prior to time zero). We assume that $G^s$ has density $g^s$ and $G^r$,
restricted to $[0,\infty)$, has density $g^r$, with~$G^r$ possibly
having some mass at infinity.
This implies, in particular, that
$G^r(0+) = G^s(0+) = 0$. We define $h^r = g^r/(1-G^r)$ and $h^s = g^s/(1-G^s)$
to be the corresponding hazard rate functions associated with
$G^r$ and $G^s$. Let\looseness=-1
\begin{eqnarray*}
H^r & \doteq& \sup\{x \in[0,\infty)\dvtx G^r(x) < 1 \}, \\
H^s & \doteq&
\sup\{x \in[0,\infty)\dvtx G^s(x) < 1 \}.
\end{eqnarray*}\looseness=0
The superscript $(N)$ will be used to refer to quantities associated
with the system with $N$ servers.

Let $\en$ denote the cumulative arrival process associated with the
system that has $N$ servers, with $\en(t)$ representing the total
number of customers that arrive into the system in the time interval
$[0,t]$. We assume that $\en$ is a renewal process with a common
interarrival distribution function $F^{(N)}$, which has finite mean and
satisfies $F^{(N)}(0)=0$. Let $\lambda^{(N)}$ be the inverse of the
mean of $F^{(N)}$, that is,
\[
\lambda^{(N)}\int_0^\infty x F^{(N)}(dx) = 1.
\]
The number $\lambda^{(N)}$ represents the long-run average arrival
rate of customers to the system with $N$ servers.
We assume $\en$, the sequence of service requirements
$\{v_j, j \in\Z\}$ and the sequence of patience times
$\{\patreq_j, j \in\Z\}$ are mutually independent. Let $\ren$ be a
c\`{a}dl\`{a}g, real-valued process defined by $\alpha_E^{(N)} (s) \doteq
\alpha_E^{(N)}(0) + s$ if
$E^{(N)} (s) = 0$ and, if $E^{(N)} (s) >0$, then
\[
\alpha_E^{(N)} (s) \doteq s - \sup \bigl\{ u < s\dvtx E^{(N)} (u) <
E^{(N)} (s) \bigr\}.
\]
Observe that the quantity $\alpha_E^{(N)}(s)$ denotes the  time to $s$
since the last arrival,
 and coincides with the backward recurrence time process. Moreover,
 the process
$\alpha_E^{(N)}$ determines the process $E^{(N)}$. Let $\ME$ be an
a.s. finite~$\Z_+$-valued random variable that represents the
number of customers that entered the system prior to time zero.
This random variable does not play an important role in the analysis.
It is
used merely for bookkeeping purposes, to keep track of the indices of customers.

\subsection{State descriptor}
\label{secrepdyn}
A Markovian description of the state of the system with $N$ servers
would require one to keep track of the residual or elapsed patience
times and the residual or elapsed service times of each customer
present in the queue or in service. In order to do this in a succinct
manner, with a common state space for all $N$-server systems, we use
the representation introduced in \cite{kanram08b}.
In this representation, the state of the $N$-server system consists of
the backward
recurrence time $\alpha_E^{(N)}$ of the renewal arrival process,
a nonnegative real-valued process $\xn$, which represents the
total number of customers in system with $N$ servers (including those
in service and those in queue)
and
a pair of measure-valued
processes, the ``age measure'' process, $\measn$,
which encodes the amounts of time
that customers currently receiving service have been in service and the
``potential queue measure'' process, $\renegn$,
which keeps track not only of the waiting times of customers in
queue but also of the potential waiting times (defined to be the
times since entry into system) of every customer (irrespective
of whether the customer has already entered service and possibly
departed the system)
for whom the potential waiting time has not yet exceeded the
patience time.
Thus, the state of the system, denoted by $\yn$, takes the form
%
%
\begin{equation}
\label{disyn}\yn= \bigl(\ren, \xn, \measn, \renegn\bigr).
\end{equation}

Note that $\xn$ and $\renegn$, together, yield the number and waiting
times~of
customers currently in queue. Indeed, for $t\in[0,\infty)$, let $\qn
(t)$ be
the~number of customers waiting in queue
at time $t$. Because the head-of-the-line~customer is the customer in
queue with the
longest waiting time, the quantity
%
%
\begin{equation}\label{def-chi}\qquad
 \chi^{(N)}(t)\doteq
\inf\bigl\{x>0\dvtx\renegn_t[0,x]\geq\qn(t)\bigr\} =
\bigl(F^{\renegn_t}\bigr)^{-1}
\bigl(\qn(t)\bigr)
\end{equation}
represents the waiting time of the head-of-the-line customer in the
queue at
time~$t$. Here, the function $F^{\eta_t^{(N)}}$ and its inverse are as
defined in (\ref{def-fmu}) and (\ref{def-finv}), respectively.
Since this is an FCFS system,
any mass in $\renegn_t$ that lies to the right of $\chi^{(N)}(t)$ represents
a customer that has already entered service by time $t$, and all masses
in $[0,\chi^{(N)}(t)]$ are still in queue.
Therefore, the queue length process $\qn$ can be expressed in terms of
$\chi^{(N)}$ and $\renegn$:
%
%
\begin{equation} \label{qn}
\qn(t)=\renegn_t\bigl[0,\chi^{(N)}(t)\bigr],\qquad t \in[0,\infty),
\end{equation}
and the restriction of $\eta_t^{(N)}$ to $[0,\chi^{(N)}(t)]$ encapsulates
the waiting times of all customers in queue at time $t$. As explained
in Section 2.2 of \cite{kanram08b}, we include in the state the
measure-valued process $\eta^{(N)}$ rather than a measure-valued
process that only keeps track of the waiting times of customers in
queue because the dynamics of the former is easier to analyze.

We note that, due to the nonidling condition, the
queue length process also admits the following alternative representation
in terms of $X^{(N)}$:
\[
\qn(t)=\bigl[\xn(t)-N\bigr]^+.
\]
Moreover, because
%
%
\begin{equation}
\label{def-xn}
\xn= \bigl\lan{\mathbf1}, \measn\bigr\ran+ \qn,
\end{equation}
the nonidling condition is equivalent to
%
%
\begin{equation}
\label{def-nonidling}
N-\bigl\lan{\mathbf1}, \measn\bigr\ran= \bigl[N - \xn\bigr]^+.
\end{equation}

The following auxiliary processes are useful for the evolution
of the system and can be recovered from the state of the system $\yn$
by using equations~(2.9)--(2.11) and (2.14) in \cite{kanram08b}:
\begin{itemize}
\item
the cumulative reneging process $\rn$, where $\rn(t)$ is the
cumulative number
of customers that have reneged from the system in the time interval $[0,t]$;
\item
the cumulative potential reneging process $\sn$, where $\sn(t)$
represents the cumulative number of customers whose potential waiting
times reached their patience times in the interval $[0,t]$;
\item
the cumulative departure process $\dn$, where $\dn(t)$ is the cumulative
number of customers that departed the system after completion of service
in the interval $[0,t]$;
\item
the process $\kn$, where $\kn(t)$ represents the cumulative number of
customers that entered service in the interval $[0,t]$.
\end{itemize}

It is easy to see from (2.16) in \cite{kanram08b} that the following
mass balance for the number of customers
in queue holds:
%
%
\begin{equation}
\label{mass-queue}
\qn(0) + \en= \qn+ \rn+ \kn.
\end{equation}

\subsection{A useful representation formula}
\label{subs-dyneq}

We now establish representation formulae (in Proposition \ref{lemde})
for expectations
of linear functionals of the age and potential queue measure-valued processes.
These are used to establish tightness of the sequence of stationary
distributions in Section \ref{secexistSD}. This representation
formula is similar to that obtained for
the fluid in Theorem 4.1 of \cite{kasram07}. The representation can be
deduced from a result given in Proposition~4.1 of \cite{kanram08b}
which, for convenience,
we first reproduce below as Proposition~\ref{threp-pde}.
\begin{prop}
\label{threp-pde} Let $G$ be the cumulative distribution function of
a~pro\-bability distribution with density $g$ and hazard rate function
$h=g/(1-G)$, let $H\doteq\sup\{x\in[0,\infty)\dvtx G(x)<1\}$. Suppose
$\overline\pi\in{\cal
D}_{\mmf}[0,\infty)$ has the property that for every $L \in[0,H)$
and $T\in
[0,\infty)$, there exists $C(L,T)<\infty$ such that
%
%
\begin{equation}\label{cond-pimeas} \int_0^\infty<\newf(\cdot
,s)h(\cdot
),\qquad\overline\pi_s>ds<C(L,T)\|\newf\|_\infty,
\end{equation}
for every $\newf\in{\cal
C}_c((-\infty,H)\times\R)$ with $\operatorname{supp}(\newf)\subset
[0,L]\times[0,T]$. Then, given any $\overline\pi_0 \in
\mmf$ and $\newfk\in\incspace$, $\overline\pi$ satisfies the
integral equation
%
%
\begin{eqnarray}\label{eq-pipde}\quad
\lan
\newft, \overline\pi_t \ran& = & \lan\newf(\cdot,0),
\overline\pi_0
\ran+ \int_0^t \lan\dtnewfs, \overline\pi_s \ran \,ds +
\int_0^t \lan\dxnewfs, \overline\pi_s \ran \,ds \nonumber\\[-8pt]\\[-8pt]
& &{} - \int_0^t \lan\newfs h(\cdot), \overline\pi_s
\ran \,ds + \int_{[0,t]} \newf(0,s) \,d \newfk(s)
\nonumber
\end{eqnarray}
for every $\newf\in{\cal C}_c((-\infty,H)\times\R)$ and $t \in
[0,\infty)$, if and only if $\overline\pi$ satisfies
%
%
\begin{eqnarray}
\label{eq-pirep}
\int_{[0,M)} f (x) \overline\pi_t (dx) & = & \int_{[0,M )} f(x+t)
\frac{1 - G(x+t)}{1 - G(x)} \overline\pi_0 (dx) \nonumber\\[-8pt]\\[-8pt]
& &{} + \int_{[0,t]} f(t-s) \bigl(1
- G(t-s)\bigr) \,d \newfk(s)\nonumber
\end{eqnarray}
for every $f \in\cb$ and $t\in(0,\infty)$.
\end{prop}

We now use this general result to obtain a useful
representation formula, which can also alternatively be deduced
by taking expectations in the representation
formula provided in Proposition 6.4 of \cite{kasram10}.
%
\begin{prop} \label{lemde} Suppose that $\E[\lan{\mathbf1},\renegn
_0\ran]<\infty$ and $\E[\lan{\mathbf1},\measn_0\ran]<\infty$. Then
for each bounded measurable function $f$ on $\R_+$ and $t\geq0$,
%
%
\begin{eqnarray} \label{dis-key1} \E\bigl[\bigl\lan f, \renegn_t \bigr\ran
\bigr] & =
& \E\biggl[\int_{[0,H^r)}f(x+t)\frac{1-G^r(x+t)}{1-G^r(x)} \renegn
_0(dx)\biggr]
\nonumber\\[-8pt]\\[-8pt]
& &{} + \E\biggl[\int_{[0,t]}f(t-s)\bigl(1-G^r(t-s)\bigr) \,d\en(s)\biggr]
\nonumber
\end{eqnarray}
and
%
%
\begin{eqnarray} \label{dis-key2}
\E\bigl[\bigl\lan f, \measn_t \bigr\ran\bigr] & = &
\E\biggl[\int_{[0,H^s)}f(x+t)\frac{1-G^s(x+t)}{1-G^s(x)} \measn
_0(dx)\biggr]
\nonumber\\[-8pt]\\[-8pt]
& &{} + \E\biggl[\int_{[0,t]} f(t-s)\bigl(1-G^s(t-s)\bigr) \,d\kn(s)\biggr].
\nonumber
\end{eqnarray}
\end{prop}
\begin{pf}
We provide the details of the proof of (\ref{dis-key1})
only, because the proof of (\ref{dis-key2}) is exactly analogous.
Fix $N \in\N$ and define $\overline{\pi} \doteq\E[\eta^{(N)}]$ and
$\overline{Z} \doteq\E[E^{(N)}]$, $G \doteq G^r$ and $h \doteq h^r$.
By Proposition \ref{threp-pde}, in
order to establish~(\ref{dis-key1}) it suffices to show that
(\ref{cond-pimeas}) and (\ref{eq-pipde}) are satisfied with $\overline
{\pi}$ and $\overline{Z}$
defined as above.
However, these are easily deduced from properties established in
\cite{kanram08b}.\vadjust{\goodbreak}
Indeed, by the analog of (5.4) of Proposition 5.1(2)
in \cite{kanram08b}, we know that
%
%
\begin{equation}
\label{rep-eq2}
\E\biggl[ \int_0^T \bigl\lan\newf(\cdot, s) h^r(\cdot),
\renegn_s \bigr\ran \,ds \biggr]
\leq C(L,T)\Vert\newf\Vert_{\infty},
\end{equation}
where $C(L,T) \doteq( \int_0^L h^r(x) \,dx ) \E[\xn
(0) +
\en(T)]$ is finite because of the supposition of the proposition, the
relation $\fx(0)\leq\lan{\mathbf1},\renegn_0\ran+\lan{\mathbf
1},\measn_0\ran$ and the fact that $\en$ is a renewal process with finite
mean.
Thus, (\ref{rep-eq2}) implies~(\ref{cond-pimeas}). On the other hand,
for every $\newf\in{\cal C}_c^1([0,H^r) \times\R_+)$,
(2.28) of Theorem~2.1 of
\cite{kanram08b} implies that for every
$t \in(0,\infty)$,
%
%
\begin{eqnarray}
\label{rep-eq1}
\bigl\lan\newf(\cdot, t), \renegn_t \bigr\ran& = &
\bigl\lan\newf(\cdot, 0), \renegn_0 \bigr\ran+
\int_0^t \bigl\lan\newf_s(\cdot,s) + \newf_x(\cdot, s), \renegn_s
\bigr\ran \,ds
\nonumber\\[-8pt]\\[-8pt]
& &{} - S^{(N)}_{\newf} (t) + \int_{[0,t]} \newf(0,s) \,d\en(s),\nonumber
\end{eqnarray}
and Proposition 5.1(2) of \cite{kanram08b} shows that
\[
M^{(N)}_{\newf,\eta} \doteq S_{\newf}^{N} - \int_0^t \bigl\lan
\newf(\cdot,s) h^r(\cdot),
\renegn_s \bigr\ran \,ds
\]
is a local $\{{\cal F}_t^{(N)}\}$ martingale.
In fact, $M^{(N)}_{\newf,\eta}$ is an
$\{{\cal F}_t^{(N)}\}$-martingale
because
\begin{eqnarray*}\E\Bigl[\sup_{s \in[0,t]} \bigl|M^{(N)}_{\newf
,\eta}(s)\bigr|\Bigr]
&\leq&\E[S_{\newf}^{N}(t)]+\E\biggl[ \int_0^T \bigl\lan
|\newf(\cdot, s)| h^r(\cdot),
\renegn_s \bigr\ran \,ds \biggr] \\ & \leq& \Vert\newf\Vert_{\infty
} \E
\bigl[\en(t)\bigr]+C(L,T)\Vert\newf\Vert_{\infty} < \infty,
\end{eqnarray*}
where the finiteness follows from the assumption that $\en$ is a renewal
process with finite mean.
The relation (\ref{eq-pipde}) then follows on
taking expectations of both sides of (\ref{rep-eq1}) and interchanging
the expectation with integration.
Hence,
the representation (\ref{dis-key1}) follows.
\end{pf}

\subsection{State space and filtration}
\label{subsub-filt}
The total number of customers
in service at time $t$ is given by
\[
\bigl\lan{\mathbf1}, \measn_t \bigr\ran=\measn_t [0,H^s)
\]
and is bounded above by the number of servers $N$. On the other hand,
it is clear (see, e.g., (2.13) of \cite{kanram08b}) that a.s., for
every $t\in[0,\infty)$,
\[
\bigl\lan{\mathbf1}, \renegn_t \bigr\ran= \renegn_t[0,H^r)
\leq\en(t) + \bigl\lan{\mathbf1}, \renegn_0 \bigr\ran\leq
\en(t) + {\cal E}_0^{(N)}<\infty.
\]
Therefore, a.s.,
for every $t\in[0,\infty)$, $\measn_t \in{\cal M}_F [0,H^s)$ and
$\renegn_t \in{\cal M}_F [0,H^r)$.

Let $\mmds$ be the subset of measures in $\mmfs$ that can be
represented as the sum of a finite number of unit Dirac measures in
$[0,H^s)$, that is, measures that take the form $\sum_{i=1}^k \delta
_{x_i}$ for some $k\in\Z_+$ and $x_i\in[0,H^s), i=1,\ldots,k$.
Analogously, let $\mmdr$ be the subset of $\mmfr$ that can be
expressed as the sum of a finite number of unit Dirac measures in
$[0,H^r)$. Also, define
%
%
\begin{eqnarray}\quad
{\cal Y}^{(N)} &\doteq&
\{(\alpha, x, \mu, \pi) \in\R_+\times\Z
_+\times\mmds\times\mmdr:\nonumber\\[-8pt]\\[-8pt]
&&\hspace*{108.3pt}x\leq\lan{\mathbf1}, \mu\ran+\lan{\mathbf1}, \pi\ran, \lan
{\mathbf1}, \mu\ran\leq N\},\nonumber
\end{eqnarray}
where $\R_+$ is endowed with the Euclidean topology $d$, $\Z_+$ is
endowed with the discrete topology $\rho$ and $\mmds$ and $\mmdr$
are both endowed with the topology of weak convergence. The space
${\cal Y}^{(N)}$ is a closed subset of $\R_+\times\Z_+\times\mmfs
\times\mmfr$ and is endowed with the usual product topology. Since
$\R_+\times\Z_+\times\mmfs\times\mmfr$ is a Polish space, the
closed subset ${\cal Y}^{(N)}$ is also a Polish space.
It follows from the representations for $\measn_t$ and $\renegn_t$
given in (2.3) and (2.8) of \cite{kanram08b} that a.s., the state
descrip\-tor~$\yn(t)$ takes values in ${\cal Y}^{(N)}$ for every $t\in
[0,\infty)$.

For $t\in[0,\infty)$, let $\tilde{{\cal F}}_t^{(N)}$ be the
$\sigma$-algebra generated by
\begin{eqnarray*}
&&\bigl\{\ME,\xn(0),\ren(s), \waitn_j(s),
\agen_j(s), s_j^{(N)},\\
&&\hspace*{44.2pt}j \in\bigl\{-\ME+1, \ldots, 0\bigr\}\cup\N, s
\in[0,t]\bigr\},
\end{eqnarray*}
where $s^{(N)} \doteq(s_j^{(N)}, j \in\Z)$ is the ``station
process,'' defined on the same probability space $(\Omega,\cal F,\PP
)$. For each $t \in[0,\infty)$, if customer $j$ has already
entered service by time $t$, then
$s_j^{(N)} (t)$ is equal to the index $i \in\{1, \ldots, N\}$ of the
station at which
customer $j$ receives service and $s_j^{(N)} (t) \doteq
0$ otherwise. Let $\{{\cal F}_t^{(N)}\}$ denote the associated
right-continuous filtration, completed with respect to $\PP$. It is
proved in Appendix A of \cite{kanram08b} that the state descriptor
$\yn$ and the auxiliary processes
$\en$, $\qn$, $\sn$, $\rn$, $\dn$ and $\kn$ are c\`{a}dl\`{a}g
and adapted to the filtration $\{{\cal F}_t^{(N)}\}$.
Moreover, from Lemma B.1 of \cite{kanram08b} it follows that
$\yn$ is a strong Markov process with respect to the filtration $\{
{\cal F}_t^{(N)}\}$.

\section{Assumptions and main results}
\label{subs-mainres}

The main focus of this paper is to
obtain a ``first-order'' approximation for the
stationary distribution of the $N$-server queue, which is accurate
in the limit as the number of servers goes to infinity.

\subsection{Basic assumptions} \label{secass}
We impose the following mild first moment assumption on the patience and
service time distribution functions $G^r$ and $G^s$. Without loss of
generality, we can
normalize the service time distribution so that its mean equals
$1$.
\begin{ass}
\label{as-mean}
The mean patience and service times are finite:
%
%
\begin{equation}
\label{def-mean1}
\theta^r \doteq\int_{[0,\infty)} x g^r(x) \,dx = \int_{[0,\infty
)} \bigl( 1 - G^r(x) \bigr) \,dx <\infty
\end{equation}
and
%
%
\begin{equation}
\label{def-mean2}
\int_{[0,\infty)} x g^s(x) \,dx = \int_{[0,\infty)} \bigl( 1 -
G^s(x) \bigr)
\,dx = 1.
\end{equation}
\end{ass}

Let $\invmeas$ and $\invrenegs$ be the probability measures defined as
follows:
%
%
\begin{eqnarray}
\label{def-invmeas}
\invmeas[0,x) & \doteq& \int_0^x \bigl(1 - G^s(y)\bigr) \,dy,\qquad x \in
[0,H^s), \\
\label{def-invrenegs}
\invrenegs[0,x) & \doteq& \int_0^x \bigl(1 - G^r(y)\bigr) \,dy,\qquad x \in[0,H^r).
\end{eqnarray}
Note that $\invmeas$ and $\invrenegs$ are well defined due to
Assumption \ref{as-mean}.
For $\lambda\geq1$, define the set $B_\lambda$ as follows:
%
%
\begin{equation}
\label{eq-fp}
B_\lambda\doteq\biggl\{x \in[1,\infty): G^r \bigl(
(F^{\lambda\invrenegs}
)^{-1} \bigl((x-1)^+\bigr) \bigr) = \frac{\lambda-
1}{\lambda} \biggr\}.
\end{equation}
Let
\[
b_l^\lambda\doteq\inf\{x\in[1,\infty)\dvtx x\in B_\lambda
\} \quad\mbox{and}\quad b_r^\lambda\doteq\sup\{
x\in[1,\infty)\dvtx x\in B_\lambda\}.
\]
Since the functions $G^r$ and $F^{\lambda\invrenegs}$ are continuous
and nondecreasing, we have $B_\lambda= [b_l^\lambda,b_r^\lambda]$.
Let ${\cal I}_{\lambda}$ be the set of states defined by
%
%
\begin{equation}
\label{eq-invman}
{\cal I}_{\lambda} \doteq
\cases{
\{ (\lambda, \lambda\invmeas, \lambda\invrenegs) \}, &
\quad if $\lambda< 1$, \cr
\{ (x_*, \invmeas, \lambda\invrenegs)\dvtx x_* \in B_\lambda
\}, &\quad
if $\lambda\geq1$.}
\end{equation}
We show in Theorem \ref{th-invman} that
${\cal I}_{\lambda}$ describes the so-called invariant manifold for the
fluid limit.
Suppose that ${\cal I}_{\lambda}$ satisfies the following assumption.
\begin{ass}\label{ass-unique}
The set ${\cal I}_{\lambda}$ has a single element.
\end{ass}

Note that this is a nontrivial restriction only when $\lambda\geq
1$. A deterministic fluid limit of the $GI/GI/N+G$ queue
was conjectured to exist in Conjecture~2.2 of~\cite{whifluid06}, and
Theorem 3.1 of \cite{whifluid06} states that this
fluid limit has a~unique steady state. However,
as shown in the example in Section \ref{subs-interchange}, in general
there need not
be a unique invariant state (or, equivalently, a unique steady state in the
sense of \cite{whifluid06}) due to the possibility of the existence of multiple
solutions to the equation (\ref{eqnGr}) below.
Thus, we explicitly
assume uniqueness of the steady state to obtain the full convergence result.
We now provide a~general
sufficient condition for Assumption \ref{ass-unique} to hold.
\begin{lemma}
\label{cor-invman}
If either $\lambda< 1$ or
$\lambda\in[1,\infty)$ and the equation
%
%
\begin{equation}G^r(x)=\frac{\lambda-1}{\lambda}
\label{eqnGr}
\end{equation}
has a unique solution, then Assumption \ref{ass-unique} holds. In
particular, this is true if $G^r$ is strictly
increasing.
\end{lemma}
\begin{pf}
Fix $\lambda\in[1,\infty)$. It suffices to show that the set
$B_\lambda$ in (\ref{eq-fp}) consists of a single point.
Since\vspace*{1pt} the equation in (\ref{eqnGr}) has a unique solution and the
function $(F^{\lambda\invrenegs}
)^{-1}(\cdot)$ is strictly increasing on $[0,\lambda\theta^r)$, the
equation
\[
G^r \bigl( (F^{\lambda\invrenegs}
)^{-1} \bigl((x-1)^+\bigr) \bigr) = \frac{\lambda-
1}{\lambda}
\]
has a unique solution. Thus, $B_\lambda$ has a single element and the
lemma follows.
\end{pf}

For each $N \in\N$, let $ \fyn= (\fren,\fxn,\fmeasn,\frenegn)$
be the fluid scaled~sta\-te descriptor defined
as follows: for $t \in[0,\infty)$ and any Borel subset $B$ of~$\R_+$,
%
%
\begin{eqnarray}
\label{fl-scaling}
\fren(t)& \doteq&\ren(t),\qquad \fxn(t) \doteq\frac{\xn(t)}{N},
\nonumber\\[-8pt]\\[-8pt]
\fmeasn_t (B) &\doteq& \frac{\measn_t (B)}{N},\qquad \frenegn_t (B) \doteq
\frac{\renegn_t (B)}{N}.\nonumber
\end{eqnarray}
Analogously, for $I=E,D,K,Q,R,S$, define
%
%
\begin{eqnarray}
\label{fl-scaling2}
\fidlen\doteq\frac{\idlen}{N}. 
\end{eqnarray}

The following standard assumption is imposed on the sequences of
fluid scaled external arrival processes $\{\fen\}$ and initial conditions
$(\renegn_0,\measn_0)$, $N \in\N$.
%
\begin{ass}
\label{ass-en} The following conditions are satisfied:
\begin{longlist}[(2)]
\item[(1)] There exists $\lambda\in[0,\infty)$ such that
$\overline{\lambda}{}^{(N)} = \lambda^{(N)}/N \rightarrow\lambda$ as
$N \rightarrow\infty$;\vspace*{1pt}
\item[(2)] As $N \ra\infty$, $\fen\ra\fe$ in ${\cal D}_{\R
_+}[0,\infty)$
$\PP$-a.s., where $\fe(t)=\lambda t$;\vspace*{1pt}
\item[(3)] $\E[\lan{\mathbf1},\renegn_0\ran]<\infty$ and $\E[\lan
{\mathbf1},\measn_0\ran]<\infty$ for each $N\in\N$.
\end{longlist}
\end{ass}

The following technical assumption was imposed on the hazard rate functions
in \cite{kanram08b} to establish the fluid
limit theorem.
\begin{ass} \label{ass-h}
There exists $L^s<H^s$ such that $h^s$ is either bounded or
lower-semicontinuous on $(L^s,H^s)$, and likewise, there exists
$L^r<H^r$ such that $h^r$ is either bounded or lower-semicontinuous on
$(L^r,H^r)$.
\end{ass}

We conclude with a mild assumption on the interarrival distribution
function $F^{(N)}$.\vadjust{\goodbreak}
\begin{ass} \label{assinterdis}
The interarrival distribution $F^{(N)}$ has a density.
\end{ass}

\subsection{Main results}
\label{subs-mainresults}

The first result focuses on the existence of a stationary distribution
for the state process.
\begin{theorem}
\label{thm-convstat0}
For each $N$, under Assumption \ref{assinterdis},
$\{\yn_t, {\cal F}_t^{(N)}\}$ is a~Feller process
that has a stationary distribution.
\end{theorem}

The Feller property is proved in Proposition \ref{lemfeller} and
the existence of a stationary distribution is established in Theorem
\ref{thm-SD}. In Theorem \ref{thm-ergodic},
the state process is also shown to be
ergodic under an additional condition (Assumption \ref{assPosF}) which
holds, for example,
when the interarrival,
reneging and service
densities are strictly positive and the latter two have support on
$(0,\infty)$.

We now state the main result, which provides a first-order approximation
for stationary distributions of $N$-server queues.
\begin{theorem}
\label{thm-convstat} $\!\!\!$Suppose Assumptions \ref{as-mean}, \ref{ass-en} and
\ref{assinterdis} hold and for \mbox{$N \in\N$}, let
$\fyn_* = (\overline{\alpha}{}^{(N)}_{E,*}, \fxns,\fmeasns,\frenegns
)$ be a scaled stationary
distribution for the $N$-server queue with abandonment.
Then the sequence $\fyn_*, N \in\N$, is tight.
If, in addition, Assumption \ref{ass-h} holds, then the limit of
any convergent subsequence of the sequence
$(\fxns, \fmeasns, \frenegns), N \in\N$,
almost surely takes values on the invariant manifold ${\cal I}_{\lambda}$.
Furthermore, if Assumption \ref{ass-unique} also holds, then
the sequence $(\fxns, \fmeasns, \frenegns)$, $N \in\N$,
converges to the unique element of ${\cal I}_{\lambda}$.
\end{theorem}

A related discrete-time result was conjectured in Theorem 7.2 of Whitt~\cite{whifluid06}.
In particular, Theorem 7.2 of
\cite{whifluid06} states that if the discrete
model introduced in \cite{whifluid06} satisfies the assumptions that
(i) each $N$-server queueing system converges (for large times) to a
unique stationary distribution; (ii) the sequence of fluid-scaled
stationary distributions is tight; and (iii) the sequence of fluid-scaled
stationary distributions has a weak limit as $N \ra\infty$,
this limit must be equal to the unique steady-state associated with the
fluid model described in \cite{whifluid06}.
The validity of properties (ii) and (iii) was not established in
\cite{manzel05}.
In contrast, we consider the continuous model,
and for this model establish tightness and (under
the  additional assumption that there is a unique invariant state)
existence of a weak limit.
The proof of Theorem \ref{thm-convstat} is given in Section \ref
{secconv} and
consists of the following main steps.
In Theorem \ref{thm-convstat0}, the Markovian
nature of the state representation is used to establish the
existence of a stationary distribution for each $N$-server system.
In Theorem~\ref{thm-tight} a convenient representation for the state
dynamics in
the $N$-server system (see Proposition \ref{lemde}) is used to\vadjust{\goodbreak}
establish tightness of
any sequence of fluid-scaled stationary distributions.
It is shown in Section \ref{subs-interchange} that, in general,
the steady state (equivalently an element of the invariant manifold)
need not in fact be unique.
Nevertheless, it is shown that any subsequential limit must
be an invariant state, and that when there is a~unique invariant
state,
the desired convergence follows.
Sufficient conditions for uniqueness of the invariant state are given
in Lemma
\ref{cor-invman}.

The characterization of the stationary distribution
and a better understanding of the possible metastable behavior
of the $N$-server queue in the presence of multiple invariant
states for the fluid remains a subject for future investigation.

\section{Stationary distribution of the $N$-server queue}
\label{secSD}

We now establish the existence of a stationary distribution for the Markovian
state descriptor
$\{\yn_t, {\cal F}_t^{(N)}\}$ for the system with $N$ servers, under
Assumption \ref{assinterdis}.
First, in Section \ref{secSMF}, $\{\yn_t,{\cal
F}_t^{(N)}\}_{t\geq0}$ is shown to be a Feller process
(see Proposition \ref{lemfeller}). Then, in Section
\ref{secexistSD}, the Krylov--Bogoliubov existence theorem (cf. Corollary
3.1.2 of \cite{GDPJZ}) is used to show that $\{\yn_t,{\cal
F}_t^{(N)}\}_{t\geq
0}$ has a stationary distribution. Finally, in the \hyperref[secPHR]{Appendix},
ergodicity and
positive Harris recurrence of the process $\{\yn_t,{\cal F}_t^{(N)}\}
_{t\geq
0}$
is established under an additional condition (Assumption
\ref{assPosF}).
For conciseness, in the rest of this section, $N$~is fixed and
the dependence on $N$ is omitted from the notation.

\subsection{Feller property} \label{secSMF}
It follows from the definition of $Y$ in (\ref{disyn}) and Lem\-ma~B.1 of
\cite{kanram08b} that $Y$ is a so-called piecewise deterministic Markov
process with jump times $\{\tau_1,\tau_2,\ldots\}$
(see \cite{jacobsen} for a precise definition of piecewise deterministic
Markov processes),
where each jump time is either the arrival time of a new customer, the
time of a service completion or the end of a~patience time.
Note that the set of jump times also includes the time of entry into service
of each customer because, due to the nonidling condition, each such
entry time
coincides with either the arrival time of that customer or the time of service
completion of another customer. Let $\tau_0=0$. For each $i \in\Z
_+$, $Y$
evolves in a deterministic fashion on $[\tau_i,\tau_{i+1})$,
\[
Y(\tau_i+t)
= \phi_{Y(\tau_i)}(t),\qquad t\in[0,\tau_{i+1}-\tau_i),
\]
where, for each
$y\in\cal Y$ of the form $y=(\alpha, x, \sum_{i=1}^k \delta
_{u_i},
\sum_{j=1}^l \delta_{z_j})$, $k, l\in\N$, $k\leq N$, we define
%
%
\begin{equation}
\label{phiF} \phi_y(t)\doteq\Biggl(\alpha+t, x,
\sum_{i=1}^k
\delta_{u_i+t}, \sum_{j=1}^l \delta_{z_j+t}\Biggr),\qquad t\geq
0.
\end{equation}
The Markovian semigroup of $Y$ is defined in the usual way: for each
$t\geq0$, $y\in{\cal Y}$ and $A\in{\cal B}({\cal Y})$, the set of
Borel subsets of ${\cal Y}$, let
%
%
\begin{equation}
\label{dissemiG} P_t(y,A)\doteq
\PP\bigl(Y(t)\in A|Y(0)=y\bigr).
\end{equation}
Moreover, for any measurable function $\psi$
defined on ${\cal Y}$ and $t\geq0$, let $P_t\psi$ be the function on
${\cal Y}$ given by
%
%
\begin{equation}\label{defPpsi} P_t\psi(y)\doteq\E[\psi
(Y(t))|Y(0)=y],\qquad y\in{\cal Y}.
\end{equation}
We now show that the semigroup $\{P_t, t\geq0\}$ is Feller in the
sense of \cite{GDPJZ} (see the beginning of Section 3.1 therein), that is,
we show that for any $\psi\in C_b({\cal Y})$ and $t\geq0$, $P_t\psi
\in C_b({\cal Y})$.

For each $m\in\Z_+$, let $Y^m$ be the state descriptor of an
$N$-server queue
with initial state
\[
Y^m(0)=y^m=\Biggl(\alpha^{m}, x^m, \sum_{i=1}^{k^m}
\delta_{u_i^m}, \sum_{j=1}^{l^m} \delta_{z_j^m}\Biggr)\in\cal Y
\]
for some $k^m \in\{0, \ldots, N\}$ and $l^m \in\N$. Suppose
that $\{Y^m, m\in\Z_+\}$ are defined on the same probability space
and $y^m$
converges to $y^0$ as $m\ra\infty$. Due to the nature of the topology
on ${\cal Y}$, the convergence of $y^m$ to $y^0$ implies
that $x^m=x^0, k^m=k^0, l^m=l^0$ for all sufficiently large $m$ and,
as $m\ra\infty$, $\alpha^{m}\ra\alpha^{0}$, $u_i^m\ra u_i^0$ and
$z_j^m\ra z_j^0$ for each $1\leq i\leq k^0, 1\leq j\leq l^0$.
Without loss of generality, we may assume that $x^m=x^0, k^m=k^0, l^m=l^0$
for every $m\in\Z_+$. For the $m$th $N$-server system, $m\in\Z_+$,
the time
since the arrival of
the last customer before time $0$ is $\alpha^m$ and hence, the random time to
the arrival of the first customer after time $0$ has distribution
function $F(\alpha^m + \cdot)/(1-F(\alpha^m))$, which has a density by
Assumption~\ref{assinterdis}.  Likewise, the distribution of
the residual patience time
of the initial customer associated with the point mass
$\delta_{z_j^m}$ has density $g^r(z_j^m+\cdot)/(1-G^r(z_j^m))$ and the
distribution of the residual service time of the initial customer associated
with the point mass $\delta_{u_i^m}$ has density
$g^s(u_i^{m}+\cdot)/(1-G^s(u_i^m))$.
For simplicity, henceforth we will denote $k^0, l^0, x^0$ simply by $k, l,
x$.
We assume that the elements of the sequence $\{Y^m, m\in\Z_+\}$ are
coupled so that:
\begin{itemize}
\item the interarrival times after the first arrival and the sequences
of service times and patience times of customers that arrive after time
$0$ are identical for each $N$-server queue $Y^m, m\in\Z_+$;
\item
the  first arrival time of a new customer in the $m$th $N$-server queue converges
to the first arrival time  in the $0$th $N$-server queue
(note that this is equivalent to the convergence of the
residual interarrival times at time zero in the corresponding systems);
\item for each $j=1,\ldots,l$, the residual patience time of the
customer associated with the point mass $\delta_{z_j^m}$ converges, as
$m\ra\infty$, to the residual patience time of the customer
associated with the point mass $\delta_{z_j^0}$;
\item for each $i=1,\ldots,k$, the residual service time of the
customer associated with the point mass $\delta_{u_i^m}$ converges, as
$m\ra\infty$, to the residual service time of the customer associated
with the point mass $\delta_{u_i^0}$.
\end{itemize}
\begin{lemma}\label{lemjumpT}
Suppose Assumption \ref{assinterdis} holds.
For each $m\in\Z_+$ and \mbox{$n\in\N$}, let $\tau^m_n$ be the $n$th jump
time of $Y^m$. Then for each $n\in\N$, $\tau^m_n$ converges to~$\tau
^0_n$ and $Y^m(\tau^m_n)$ converges in $\cal Y$ to $Y^0(\tau^0_n)$
a.s., as $m\ra\infty$.
\end{lemma}
\begin{pf}
We prove the lemma by an induction argument.
First, consider \mbox{$n=1$}. For each $m\in\Z_+$, the first jump time $\tau
^m_1$ is
the minimum of the first arrival time of a new customer, the residual
patience times of initial customers with potential waiting
times in the set $\{z_j^m, 1\leq j\leq l\}$ and the residual service times
of initial customers
associated with ages in the set $\{u_i^m, 1\leq i\leq k\}$. It follows
directly from the assumptions on $\{Y^m, m\in\Z_+\}$ that for every
realization,
%
%
\begin{equation}
\label{jump1conv}\tau^m_1 \ra\tau^0_1, \qquad\mbox
{as }
m\ra\infty.
\end{equation}
Since the
interarrival distribution $F$, the service time
distribution function~$G^s$ and the patience
time distribution function $G^r$ are independent and have densities,
with probability $1$,
$\tau_1^0$ coincides with exactly one of the following in the $0$th
system: the first arrival time of a new customer, the residual
patience time of an initial customer with
initial
waiting time~$z_j^0$, $1\leq j\leq l$, or the residual service time
of an initial customer with age~$u_i^0$, $1\leq i\leq k$. Let us fix a
realization such that $\tau_1^0$ is equal to the
first arrival time of a~new customer in the $0$th system. The
remaining\vspace*{1pt} two cases can be handled similarly.
In this case, by the convergence of $\tau_1^{m}$ to $\tau_{1}^{0}$, the
convergence of the other quantities stated above and the coupling
construction, for all sufficiently large $m$, $\tau_1^m$ is equal to the first
arrival time of a new customer in the $m$th system.
Hence, for all sufficiently large $m$,
the first jump of $Y^m$ is due to the first arrival of a new customer
in the
$m$th system. For such~$m$, since~$Y^m$ evolves in a deterministic
fashion on
$[0,\tau_{1}^m)$ described by the continuous function $\phi$
introduced in (\ref{phiF}),
we have
\[
Y^m(\tau^m_1-)=\Biggl(\alpha^m+\tau_1^m, x, \sum_{i=1}^k
\delta_{u_i^m+\tau^m_1}, \sum_{j=1}^l \delta_{z_j^m+\tau
^m_1}\Biggr).
\]
If
$k = N$ and $x\geq k=N$, then all the servers are busy and the customer
that arrives
at $\tau^m_1$ will have to wait in queue.
Thus, by the coupling construction,
\[
Y^m(\tau^m_1)=\Biggl(0, x+1, \sum_{i=1}^k \delta_{u_i^m+\tau
^m_1},
\sum_{j=1}^l \delta_{z_j^m+\tau^m_1}+\delta_0\Biggr).
\]
On the
other hand,
if $k < N$, then $x=k$ and there is at least one idle server present.
Hence, the customer will join service immediately upon arrival at time
$\tau^m_1$. Thus, in this case,
\[
Y^m(\tau^m_1)=\Biggl(0, x+1, \sum_{i=1}^k \delta_{u_i^m+\tau
^m_1}+\delta_0, \sum_{j=1}^l \delta_{z_j^m+\tau^m_1}+\delta
_0\Biggr).
\]
In both cases, for the chosen realization, we have $Y^m(\tau^m_1)\ra
Y^0(\tau^0_1)$ as $m\ra\infty$.

Now, suppose that $\tau^m_i$ converges to $\tau^0_i$ and $Y^m(\tau^m_i)$
converges to $Y^0(\tau^0_i)$ a.s., as $m\ra\infty$, for $1\leq i\leq
n$, and
consider $i=n+1$.
Fix a realization such that $\tau^m_n$ converges to $\tau^0_n$ and
$Y^m(\tau^m_n)$ converges to $Y^0(\tau^0_n)$ as $m\ra\infty$.
By the same argument as in the case $n=1$, we may assume, without loss
of generality, that for the chosen realization and $m\in\Z_+$, the
jump at $\tau_n^m$ for~$Y^m$ is due to the arrival of a new customer.
Then, for each $m\in\Z_+$, $ Y^m(\tau^m_n)$ has the following
representation:
\[
Y^m(\tau^m_n)= \Biggl(0, x^m_n, \sum_{i=1}^{k_n^m} \delta
_{u_{i,n}^m}, \sum_{j=1}^{l_n^m} \delta_{z_{j,n}^m}\Biggr)
\]
for some\vspace*{1pt} $k^m_n,l^m_n,x^m_n\in\Z_+$, $u_{i,n}^m,z_{j,n}^m \in
\R_+$ with $x^m_n\leq k^m_n+l^m_n$, $k^m_n\leq N$.
Due to the induction hypothesis and the topology of $\cal Y$, for all
sufficiently large $m$, $x_n^m=x_n^0$, $k_n^m=k_n^0$, $l_n^m=l_n^0$,
$u_{i,n}^m \ra u_{i,n}^0$ and $z_{j,n}^m \ra
z_{j,n}^0$ as $m\ra\infty$ for
each $1\leq i \leq k_n^0$ and $1\leq j \leq l_n^0$. The argument that
was used
for the case $n=1$
can be used again to show that $\tau^m_{n+1}$ converges to $\tau
^0_{n+1}$ and $Y^m(\tau^m_{n+1})$ converges to $Y^0(\tau^0_{n+1})$
a.s., as $m\ra\infty$. This completes the induction argument and
hence, proves the lemma.
\end{pf}
\begin{prop} \label{lemfeller}
Suppose that the interarrival distribution $F$ has a~density. Then the
semigroup $\{P_t, t\geq0\}$ is Feller.
\end{prop}
\begin{pf}
It is easy to see from the definition of the function $P_t\psi$ in~(\ref{defPpsi})
that when $\psi$ is bounded, $P_t\psi$ is also bounded. To prove the
proposition, it
suffices to show that $P_t\psi$ is a continuous function with respect
to the
topology on ${\cal Y}$. Fix $t\geq0$. Let $y^0=(\alpha^{0}, x^0, \mu
^0, \pi^0)\in\cal Y$ and $y^m=(\alpha^{m}, x^m, \mu^m,
\pi^m)$, $m\in\Z_+$, be points in $\cal Y$ such that, as
$m\rightarrow\infty$, $y^m$
converges in $\cal Y$ to $y^0$. Since $\Z_+$ is a discrete
space and $x^m\rightarrow x^0$ as $m\rightarrow\infty$, it must be
that for
all sufficiently large $m$, \mbox{$x^m=x^0$}. Without loss of generality, we assume
that $x^m=x^0$ for each $m\in\N$. Consider a sequence of coupled $N$-server
queues $\{Y^m, m\in\Z_+\}$ carried out earlier such that
$Y^m(0)=y^m$ for each $m\in\Z_+$. Then
$P_t\psi(y^m)=\E[\psi(Y^m(t))]$. To prove the continuity of~$P_t\psi
$, it
suffices to show that $Y^m(t)\ra Y^0(t)$ a.s., as $m\ra\infty$.
Indeed, since
$\psi\in C_b({\cal Y})$, the latter convergence would imply that
$\psi(Y^m(t))\rightarrow\psi(Y^0(t))$ and hence, by the bounded convergence
theorem, that $P_t\psi(y^m)\rightarrow P_t\psi(y^0)$ as $m\rightarrow
\infty$, which would show that $\{P_t, t\geq0\}$ is Feller.

It only remains to prove that almost surely, $Y^m(t)\ra Y^0(t)$ as
$m\ra\infty$.
Since the interarrival distribution $F$, service distribution $G^s$ and
patience distribution $G^r$ all have densities, with probability one
$t$ does
not belong to the set $\{\tau^0_n, n\in\N\}$ of jump times of
$Y^0$. Fix a
realization such that $t$ does not belong to the set $\{\tau^0_n,
n\in\N\}$
and such that for each $n\in\N$, $\tau^m_n$ converges to $\tau^0_n$ and
$Y^m(\tau^m_n)$ converges in $\cal Y$ to $Y^0(\tau^0_n)$, as $m\ra
\infty$.
By Lemma~\ref{lemjumpT}, this can be done on a set of probability one.
Let $r\doteq\sup\{n\dvtx\tau_n^0<t\}$. Then $\tau_r^0<t <\tau
^0_{r+1}$ and hence,
for all sufficiently large $m$, $\tau_r^m<t <\tau^m_{r+1}$. By the convergence
of $\tau_r^m$ to $\tau_r^0$ and $Y^m(\tau^m_r)$ to $Y^0(\tau^0_r)$,
as $m\ra\infty$,
as well as the definition of $\phi$ in (\ref{phiF}), we conclude that
$Y^m(t)\ra
Y^0(t)$, as $m\ra\infty$.
Thus, we have shown that $Y^m(t)\ra Y^0(t)$ a.s., as $m\ra\infty$.
\end{pf}

\subsection{Existence of stationary distributions}
\label{secexistSD}

In this section, it is shown that the Feller process $\{Y_t,{\cal
F}_t\}_{t\geq0}$ admits a stationary distribution. To achieve this,
we apply the Krylov--Bogoliubov theorem (cf. Corollary 3.1.2 of
\cite{GDPJZ})
which requires showing that the following family $\{L_t, t\geq0\}$ of
probability measures associated with $\{Y_t,{\cal F}_t\}_{t\geq0}$ is
tight. For each measurable set $B\subset{\cal Y}$ and $t>0$, define
\[
L_t(B)
\doteq\frac{1}{t}\int_0^t \PP\bigl(Y(s)\in B\bigr) \,ds.
\]
Obviously, for each $t\geq
0$, $L_t$ is a probability measure on $(\cal Y, {\cal B}(\cal Y))$.
We now recall some useful criteria for tightness of a family of random
measures, which can be derived from A7.5 of \cite{Kal} (see also
Exercise 4.11 of \cite{Kal}).
\begin{prop} \label{proptight}
A family $\{\pi_{t}\}_{t\geq0}$ of ${\cal M}_F [0,H)$-valued random
variables is tight if the following two conditions hold:
\begin{longlist}[(2)]
\item[(1)]$\sup_{t\geq0} \E[\lan{\mathbf1}, \pi_t \ran]< \infty$;
\item[(2)]$\lim_{c\rightarrow H} \sup_{t\geq0} \E[\pi_t[c,H)]
\rightarrow0$.
\end{longlist}
\end{prop}
%
%
\begin{lemma}
\label{lem-cond1} Suppose Assumption \ref{as-mean} holds and
$\E[\lan{\mathbf1}, \reneg_0 \ran] < \infty$.
Then
$\sup_{t \geq0} \E[\lan{\mathbf1}, \reneg_t \ran] < \infty$
and
$\sup_{t \geq0} \E[\lan{\mathbf1}, \meas_t \ran] < \infty$.
\end{lemma}
\begin{pf}
Let $f={\mathbf1}$ in (\ref{dis-key1}) and (recalling that the
superscript $N$ is
being suppressed from the notation)
let $e(t)\doteq\E[E(t)]$, $t\geq0$. Using
integration-by-parts, it
follows that
\begin{eqnarray*} \E[\lan{\mathbf1}, \reneg_t \ran] &
\leq
& \E[\lan{\mathbf1}, \reneg_0 \ran] + \int
_0^t\bigl(1-G^r(t-s)\bigr) \,de(s)
\\
&=& \E[\lan{\mathbf1}, \reneg_0 \ran] + e(t)-
\int_0^te(s)g^r(t-s) \,ds \\
&=&
\E[\lan{\mathbf1}, \reneg_0 \ran] + e(t)\bigl(1-G^r(t)\bigr)-
\int_0^t\bigl(e(t)-e(t-s)\bigr)g^r(s) \,ds.
\end{eqnarray*}
Since $E$ is a renewal process with rate $\lambda$,
$e(t)/t \ra\lambda$ as $t\ra\infty$ by the key renewal theorem.
Moreover, the finite mean condition (\ref{def-mean1}) implies
$t(1-G^r(t))\ra
0$ as $t\ra\infty$. Therefore, we have $\sup_{t\geq0} e(t)(1-G^r(t))
<\infty$. The Blackwell\vspace*{1pt} renewal theorem (cf. Theorem 4.3 of
\cite{asmbook})
implies that $e(t)-e(t-s) \ra s \lambda$ as $t\ra\infty$ and hence, that
$\sup_{t\geq0}\int_0^t(e(t)-e(t-s))g^r(s)\,ds<\infty$. Combining these
relations with (3) of Assumption \ref{ass-en} and the last display, we
conclude that $\sup_{t\geq0} \E[\lan{\mathbf1}, \reneg_t
\ran
]<\infty$.

On the other hand, since each $\meas_t$ is the sum of at most $N$ unit Dirac
masses, it
trivially follows that $\sup_{t\geq0}\E[\lan{\mathbf1},
\meas_t \ran]
\leq N < \infty$.\vspace*{-3pt}
\end{pf}

To show that $\{\reneg_t\}_{t\geq0}$ and $\{\meas_t\}_{t\geq0}$
satisfy the second property in Proposition \ref{proptight}, note
that by choosing $f=\indd_{[c,H^r)}, c>0$, in (\ref{dis-key1}),
we obtain for $t \geq0$,
%
%
\begin{eqnarray}
\label{dis4-1}\E[\reneg_t[c,H^r) ] &
\leq&
\E\biggl[\int_{[0,H^r)}\indd_{[c,H^r)}(x+t)\frac
{1-G^r(x+t)}{1-G^r(x)} \reneg_0(dx)\biggr]
\nonumber\\[-9.5pt]\\[-9.5pt]
& &{} + \int_0^t\indd_{[c,H^r)}(t-s)\bigl(1-G^r(t-s)\bigr) \,de(s) \nonumber
\end{eqnarray}
and, likewise, by choosing $f=\indd_{[c,H^s)}$ in (\ref{dis-key2})
it follows that for $t \geq0$,
%
%
\begin{eqnarray}
\label{dis-4-meas}\E[\meas_t[c,H^s)] &
= &
\E\biggl[\int_{[0,H^s)}\indd_{[c,H^s)}(x+t)\frac
{1-G^s(x+t)}{1-G^s(x)} \meas_0(dx)\biggr]
\nonumber\\[-9.5pt]\\[-9.5pt] & &{} +
\E\biggl[\int_0^t\indd_{[c,H^s)}(t-s)\bigl(1-G^s(t-s)\bigr) \,dK(s)\biggr].
\nonumber
\end{eqnarray}

We now establish two supporting lemmas.
%
\begin{lemma} \label{lem-elem2}
Suppose Assumption \ref{as-mean} holds and
$\E[\lan{\mathbf1}, \reneg_0 \ran] < \infty$.
We have
%
%
\begin{equation}\label{dis4-2}\lim_{c\ra H^r}\sup_{t\geq
0}\E\biggl[\int_{[0,H^r)}\indd_{[c,H^r)}(x+t)\frac
{1-G^r(x+t)}{1-G^r(x)}\reneg_0(dx)\biggr]
= 0
\end{equation}
and
%
%
\begin{equation}\label{dis4-4}\lim_{c\ra H^s}\sup_{t\geq0}\E
\biggl[\int
_{[0,H^s)}\indd_{[c,H^s)}(x+t)\frac{1-G^s(x+t)}{1-G^s(x)}\meas
_0(dx)\biggr] = 0.
\end{equation}
\end{lemma}
\begin{pf}
When $H^r<\infty$, we have
\begin{eqnarray*}
&&\sup_{t\geq0}
\E\biggl[\int_{[0,H^r)}\indd_{[c,H^r)}(x+t)\frac
{1-G^r(x+t)}{1-G^r(x)}
\reneg_0(dx)\biggr]\\[-2pt]
&&\qquad\leq\sup_{t\geq0}
\E\biggl[\int_{[0,H^r)}\indd_{[c,H^r)}(x+t) \reneg_0(dx)\biggr]
\\[-2pt]
&&\qquad= \sup_{t\in[0,c)}\E\biggl[\int_{[0,H^r)}\indd
_{[c,H^r)}(x+t)
\reneg_0(dx)\biggr] \\[-2pt]
&&\qquad\quad{}\vee\sup
_{t\in
[c,H^r)}\E\biggl[\int_{[0,H^r)}\indd_{[c,H^r)}(x+t)
\reneg_0(dx)\biggr].\vadjust{\goodbreak}
\end{eqnarray*}
Using $\E[\lan{\mathbf1}, \reneg_0 \ran] < \infty$
to justify the application of the dominated convergence theorem, we obtain
\[
\lim_{c\ra H^r}\sup_{t\in[c,H^r)}
\E\biggl[\int_{[0,H^r)}\indd_{[c,H^r)}(x+t)\reneg_0(dx)\biggr]
\leq\lim_{c\ra H^r} \E\bigl[\reneg_0[0,H^r-c)\bigr] = 0.
\]
On the other hand, we know that
\begin{eqnarray*}
&&\sup_{t\in[0,c)}\E\biggl[\int_{[0,H^r)}\indd_{[c,H^r)}(x+t)\reneg
_0(dx)\biggr]\\
&&\qquad\leq\sup_{t\in[0,c)}\E[\reneg_0(c-t,H^r-t)].
\end{eqnarray*}
We show by contradiction that $\sup_{t\in
[0,c)}\E[\reneg_0(c-t,H^r-t)]\ra0$ as $c\ra H^r$. Suppose
this is not true. Then there exist $\delta>0$ and sequences $\{c_n\}
_{n\in
\N}$ and $\{t_n\}_{n\in\N}$ such that $c_n\ra H^r$ as $n\ra\infty$,
$t_n\in[0,c_n)$ for each $n\in\N$ and
$\E[\reneg_0(c_n-t_n,H^r-t_n)] > \delta$ for each $n\in
\N$. Because we are considering the case $H^r < \infty$, $\{t_n\}
_{n\in\N}$ is bounded and so we can take a subsequence,
which we call again $\{t_n\}_{n\in\N}$, such that $\lim_{n\ra
\infty}t_n=t_*\in[0,H^r]$. In turn, this implies
\[
\lim_{n\ra\infty}\E[\reneg_0(c_n-t_n,H^r-t_n)] = 0,
\]
which contradicts the initial hypothesis. Thus,
$\sup_{t\in[0,c)}\E[\reneg_0(c-t,H^r-t)]\ra0$.
Together with the last three displays, this implies
that (\ref{dis4-2}) holds when $H^r < \infty$.
On the other hand, when $H^r=\infty$ we have
\begin{eqnarray*}
&&\sup_{t\geq0}
\E\biggl[\int_{[0,H^r)}\indd_{[c,H^r)}(x+t)\frac
{1-G^r(x+t)}{1-G^r(x)}
\reneg_0(dx)\biggr]\\
&&\qquad\leq\max\biggl(\sup_{t\in
[0,c/2)}\E\biggl[\int_{[0,\infty)}\indd_{[c,\infty)}(x+t)
\reneg_0(dx)\biggr], \\
&&\hspace*{59.6pt}\sup_{t\in[c/2,\infty)}\E\biggl[\int_{[0,\infty)} \frac
{1-G^r(x+t)}{1-G^r(x)}
\reneg_0(dx)\biggr] \biggr)\\
&&\qquad\leq\E[\reneg_0(c/2,\infty)] \vee
\E\biggl[\int_{[0,\infty)} \frac{1-G^r(x+c/2)}{1-G^r(x)}
\reneg_0(dx)\biggr].
\end{eqnarray*}
Sending $c \ra\infty$ on both sides, and using the
fact that $\E[\lan{\mathbf1}, \reneg_0 \ran] < \infty$, an
application of the dominated convergence theorem shows that the right-hand
side
vanishes and thus (\ref{dis4-2}) holds in this case too.
The proof of (\ref{dis4-4}) is exactly analogous and is thus omitted.
\end{pf}
\begin{lemma} \label{lem-elem} Suppose Assumption \ref{as-mean}
holds and
let $e(t)\doteq\E[E(t)], t\geq0$. For $(H,G)=(H^r,G^r)$ and
$(H,G)=(H^s,G^s)$, we have
%
%
\begin{equation}\label{dis4-3} \lim_{c\ra H}\sup
_{t\geq0}
\int_0^t\indd_{[c,H)}(t-s)\bigl(1-G(t-s)\bigr) \,d e(s) = 0.
\end{equation}
\end{lemma}
\begin{pf}
$E$ is a (delayed) renewal process with rate $\lambda$ and due to
Assumption~\ref{as-mean} and Proposition 4.1 in Chapter V of \cite{asmbook},
the function $x\mapsto\indd_{[c,H)}(x)(1-G(x))$ is directly Riemann
integrable.
Thus, by the key renewal theorem (cf. Theorem 4.7 of \cite{asmbook})
we obtain
\[
\lim_{t\ra\infty}\int_0^t\indd_{[c,H)}(t-s)\bigl(1-G(t-s)\bigr) \,de(s) =
\frac{1}{\lambda}\int_{[0,\infty)} \indd_{[c,H)}(x)\bigl(1-G(x)\bigr) \,dx.
\]
Since the integrability condition imposed in Assumption \ref{as-mean}
implies that
$\int_{[0,\infty)} \indd_{[c,H)}(x)(1-G(x)) \,dx \ra0$ as $c\ra
H$, we
have the desired result.
\end{pf}
\begin{lemma} \label{lemrenegT}
Suppose Assumption \ref{as-mean} holds and the initial condition satisfies
$\E[\lan{\mathbf1}, \reneg_0] < \infty$. Then
the family $\{\reneg_t\}_{t \geq0}$ of ${\cal M}_F [0,H^r)$-valued random
variables and the family $\{\meas_t\}_{t \geq0}$ of ${\cal M}_F
[0,H^s)$-valued random variables
are tight.
\end{lemma}
\begin{pf}
Both families satisfy the first condition of Proposition \ref{proptight}
due to Lemma \ref{lem-cond1}.
Combining (\ref{dis4-1}) with
(\ref{dis4-2}) and Lemma \ref{lem-elem} for the case
$(H,G)=(H^r,G^r)$, it follows that $\{\reneg_t\}_{t \geq0}$ also satisfies
the second condition of Proposition~\ref{proptight} and is thus tight.

It only remains to show that $\{\meas_t\}_{t \geq0}$ also satisfies
the second condition of
Proposition \ref{proptight}. For this, it suffices to show
that as $c \ra H^s$, the supremum (over $t$) of the right-hand side of
(\ref{dis-4-meas}) goes to zero.
Now, let $k(t) \doteq\E[K(t)]$ for $t\geq0$.
Applying the integration-by-parts and change of variable formulae to the
second term on the right-hand side of (\ref{dis-4-meas}), we see that
%
%
\begin{eqnarray}\label{dis-5}\quad
&&\sup_{t\geq
0}\E\biggl[\int_0^t\indd_{[c,H^s)}(t-s)\bigl(1-G^s(t-s)\bigr)\,dK(s)\biggr] \nonumber\\
&&\qquad=\sup_{t>c}\int_0^t\indd_{[c,H^s)}(t-s)\bigl(1-G^s(t-s)\bigr)
\,dk(s)\nonumber\\
&&\qquad= \sup_{t>c} \biggl( k(t-c)\bigl(1-G^s(c)\bigr)- k\bigl((t-H^s)^+\bigr)
\bigl(1-G^s(t\wedge H^s) \bigr) \\
&&\qquad\quad\hspace*{146pt}{} - \int_c^{t\wedge
H^s}k(t-s)g^s(s) \,ds \biggr)\nonumber\\
&&\qquad\leq\sup_{t>c}
\biggl(k(t-c)\bigl(1-G^s(t)\bigr)+\int_c^{t\wedge H^s}\bigl(k(t-c)-k(t-s)\bigr)g^s(s)
\,ds\biggr).\nonumber
\end{eqnarray}
Taking expectations of both sides of (\ref{mass-queue}), we obtain for
each $t\geq0$,
\[
\E[Q(0)] + e(t) = \E[Q(t)] + \E
[R(t)] + k(t).
\]
Since $Q$ and $R$ are nonnegative and $R$ is increasing, it follows
that
\[
k(t-c) \leq e(t-c) +\E[Q(0)]
\]
and
\[
k(t-c)-k(t-s) \leq e(t-c)-e(t-s)+\bigl(\E[Q(t-s)]-\E
[Q(t-c)]\bigr).
\]
Substituting these inequalities into (\ref{dis-5}) and
carrying out another integration-by-parts, we obtain
%
%
\begin{eqnarray}\label{add-2}
&&\sup_{t>c}
\int_0^t\indd_{[c,H^s)}(t-s)\bigl(1-G^s(t-s)\bigr) \,dk(s)\nonumber\\
&&\qquad\leq
\sup_{t>0}\int_0^t\indd_{[c,H^s)}(t-s)\bigl(1-G^s(t-s)\bigr)
\,de(s)\nonumber\\[-8pt]\\[-8pt]
&&\qquad\quad{} +
\sup_{t>c}\E[Q(0)] \bigl(1-G^s(t)\bigr)\nonumber\\
&&\qquad\quad{}+\sup_{t>c}\int_c^{t\wedge
H^s}\bigl(\E[Q(t-s)]-\E[Q(t-c)]\bigr)g^s(s) \,ds.\nonumber
\end{eqnarray}
Applying Lemma \ref{lem-elem}, with $(H,G)=(H^s,G^s)$, we have
\[
\lim_{c\ra
H^s}\sup_{t\geq0}\int_0^t\indd_{[c,H^s)}(t-s)\bigl(1-G^s(t-s)\bigr) \,de(s) = 0.
\]
Moreover,
\[
\lim_{c\ra H^s}\sup_{t>c}\E[Q(0)] \bigl(1-G^s(t)\bigr) =
\E[Q(0)]\lim_{c\ra H^s}\bigl(1-G^s(c)\bigr) = 0.
\]
Also, since $Q(t)\leq\lan{\mathbf1}, \reneg_t \ran$ by (\ref{qn}),
we have
\[
\sup_{t>c}\int_c^{t\wedge
H^s}\bigl(\E[Q(t-s)]-\E[Q(t-c)]\bigr)g^s(s) \,ds
\leq2 \sup_{t\geq0} \E[\lan{\mathbf1}, \reneg_t \ran
] \bigl(1-G^s(c)\bigr).
\]
Since Lemma \ref{lem-cond1} implies $\sup_{t \geq
0} \E[\lan{\mathbf1}, \reneg_t \ran] < \infty$,
the right-hand side of the above inequality tends to zero as $c\ra H^s$.
Combining the last five assertions with (\ref{dis-5}) and (\ref
{add-2}), it follows that
as $c \ra H^s$, the supremum over $t \geq0$ of the second term on the
right-hand side of (\ref{dis-4-meas}) vanishes to zero. On the other hand,
as $c \ra H^s$, the supremum over $t \geq0$ of the first term on the
right-hand side of (\ref{dis-4-meas}) also vanishes to zero by (\ref{dis4-4}).
Thus, we have shown that $\sup_{t\geq0}\E[\meas
_t[c,H^s)] \ra0$ as $c\ra
H^s$, and the proof of the lemma is complete.
\end{pf}
%
%
\begin{lemma} \label{lemLtight}
Suppose Assumption \ref{as-mean} holds and $\E[ \lan{\mathbf1},
\reneg_0 \ran] <
\infty$.
The family of probability measures $\{L_t\}_{t\geq0}$ is tight.
\end{lemma}
\begin{pf}
By Lemma \ref{lemrenegT}, we know that for each $\delta>0$, there
exist two compact subsets $\tilde C_{\delta}\subset{\cal M}_F
[0,H^s)$ and
$\tilde D_{\delta} \subset{\cal M}_F [0,H^r)$ such that
%
%
\begin{eqnarray}
\label{add-3}
\inf_{t\geq0}\PP(\meas_t\in\tilde C_{\delta} ) &
\geq
& 1-\delta/2, \nonumber\\[-8pt]\\[-8pt]
\inf_{t\geq0}\PP(\reneg_t\in\tilde D_{\delta} )&
\geq
& 1-\delta/2.\nonumber
\end{eqnarray}
It follows from (\ref{qn}) and (\ref{def-xn}) that $X(t)\leq\lan
{\mathbf1},
\meas_t \ran+ \lan{\mathbf1}, \reneg_t \ran$ for each $t\geq0$.
Together with
(\ref{add-3})
and the fact that the map $\mu\rightarrow\lan{\mathbf1}, \mu\ran$
is continuous,
this implies that there exists $b>0$ such that
%
%
\begin{equation}\label{compact3} \inf_{t\geq0} \PP\bigl(X(t)\leq b
\bigr)\geq
1-\delta.
\end{equation}
On the other hand, by Theorem 4.5 in Chapter V of \cite{asmbook}, it
follows that as $t\rightarrow\infty$, $\re(t)$ converges weakly to
the distribution
%
%
\begin{equation}\label{deff0} F_0(t)\doteq\lambda\int_0^t
\bigl(1-F(y)\bigr)\,dy.
\end{equation}
Thus, there exist $T_0>0$ and $c>0$ such that for all
$t\geq T_0$,
\[
\PP\bigl(\re(t)\leq a\bigr)\geq F_0(a)-\delta/2 \geq1-\delta.
\]
By choosing $a$ large enough, we may assume without loss of generality,
that
\[
\inf_{t\in[0,T_0]}\PP\bigl(\re(t)\leq a\bigr)\geq1-\delta.
\]
Define $C_\delta\doteq[0,a] \times[0, b] \times\tilde C_{\delta}
\times\tilde D_{\delta}$. Then the set $C_\delta$ is compact and
$L_t(C_\delta)\geq1-\delta$ for each $t\geq0$, which proves the
lemma.
\end{pf}

Since $\{Y_t, {\cal F}_t\}_{t\geq0}$ is a Feller process by
Proposition \ref{lemfeller}, and
Lemma \ref{lemLtight} is applicable when
the initial condition satisfies $\E[ \lan{\mathbf1}, \reneg_0 \ran
] <
\infty$,
the Krylov--Bogoliubov theorem
immediately yields the following result.
\begin{theorem} \label{thm-SD}
Suppose that Assumptions \ref{as-mean} and \ref{assinterdis} hold.
Then the state descriptor $(\re, X, \meas, \reneg)$ has a stationary
distribution $(\alpha_{E,*}, X_*, \meas_*, \reneg_*)$ that satisfies
$\E[ \lan{\mathbf1}, \reneg_* \ran] <
\infty$.
\end{theorem}

\section{Fluid limit}
\label{sec-Fluid}

In Section \ref{sec-flconj}, we describe a deterministic dynamical
system that
was shown in Theorems 3.5 and 3.6 of
\cite{kanram08b} to arise as the so-called fluid
limit of a many-server queue with abandonment that has
service time and patience time distribution functions $G^s$ and $G^r$,
respectively. In Section~\ref{secIM}, we identify the invariant manifold
associated with the fluid limit, which is then used in Section~\ref
{secconv} to
obtain a first-order asymptotic approximation to the stationary distribution
of the fluid scaled state descriptor~$\overline Y{}^{(N)}$.

\subsection{Fluid equations}
\label{sec-flconj}

The state of the fluid system at time $t$ is represented by
the triplet
\[
(\fx(t), \fmeas_t, \freneg_t) \in
\R_+ \times\mmfs\times\mmfr.
\]
Here, $\fx(t)$ represents the mass (or, equivalently, limiting scaled number
of customers) in the
system at time $t$, $\fmeas_t[0,x)$ represents the mass
of customers in service at time $t$ who have been in
service for less than $x$ units of time, whereas
$\freneg_t[0,x)$ represents the mass of customers
in the system who, at time $t$,\vadjust{\goodbreak}
have been in the system no more than $x$ units of time and
whose patience time exceeds their time in system (which implies,
in particular, that they have not yet abandoned the system).
The inputs to the system are the (limiting) cumulative arrival process
$\fe$
and the initial conditions $\fx(0)$, $\fmeas_0$ and $\freneg_0$.
Thus, $\lan{\mathbf1}, \fmeas_0 \ran$ represents the total mass of customers
in service at time $0$ and the fluid analog of the nonidling condition
(\ref{def-nonidling}) is
%
%
\begin{equation}\label{fnonidling} 1 - \lan{\mathbf1}, \fmeas_0
\ran= [1-\fx(0)]^+.
\end{equation}
The quantity $\lan{\mathbf1}, \freneg_0 \ran$ represents the
total mass of customers
at time $0$ whose residual patience times are positive. Hence, we have
\[
[\fx(0)-1]^+ \leq\lan{\mathbf1}, \freneg_0 \ran.
\]
Thus, the space of possible input data for the fluid equations is
given by
%
%
\begin{eqnarray}
\label{def-newspace}\quad
\newspace&\doteq&
\{(e,x,\nu, \eta) \in\incspace\times\R_+ \times
\mmfs\times\mmfr\dvtx\nonumber\\[-8pt]\\[-8pt]
&&\hspace*{97.3pt}1 - \lan{\mathbf1}, \nu\ran= [1-x]^+, [x-1]^+\leq\lan{\mathbf
1}, \eta\ran\},\nonumber
\end{eqnarray}
where recall that $\incspace$ is the subset of nondecreasing functions
$f \in{\cal D}_{\R_+}[0,\infty)$ with $f(0) = 0$. Let $\fntf(x)$
denote $\freneg_t[0,x]$ for each $x \in[0,H^r)$.
\begin{defn}[(Fluid equations)]
\label{def-fleqns}
Given any $(\fe,\fx(0),\fmeas_0, \freneg_0)
\in\newspace$, we say that the
c\`{a}dl\`{a}g function
$(\fx, \fmeas, \freneg)$ taking values
in $\R_+ \times\mmfs\times\mmfr$
satisfies the associated fluid equations if for every
$t \in[0,\infty)$,
%
%
\begin{equation}
\label{cond-radon}
\int_0^t \lan h^r, \freneg_s \ran \,ds < \infty,\qquad \int_0^t
\lan h^s,
\fmeas_s \ran \,ds < \infty
\end{equation}
for every bounded Borel measurable function $f$ defined on $\R_+$,
%
%
\begin{eqnarray}
\label{eqnnu}
\int_{[0,H^s)} f (x) \fmeas_t (dx) & = & \int_{[0,H^s )} f(x+t)
\frac{1 - G^s(x+t)}{1 - G^s(x)} \fmeas_0 (dx) \nonumber\\[-7pt]\\[-7pt]
& &
{}+ \int_0^t f(t-s) \bigl(1
- G^s(t-s)\bigr) \,d \fk(s)\nonumber
\end{eqnarray}
and
%
%
\begin{eqnarray}
\label{eqneta}
\int_{[0,H^r)} f (x) \freneg_t (dx) & = & \int_{[0,H^r )} f(x+t)
\frac{1 - G^r(x+t)}{1 - G^r(x)} \freneg_0 (dx) \nonumber\\[-7pt]\\[-7pt]
&&{} + \int_0^t f(t-s) \bigl(1
- G^r(t-s)\bigr) \,d \fe(s),\nonumber
\end{eqnarray}
where
%
%
\begin{eqnarray}
\label{eq-fk}
\fk(t) & = & [ \fx(0)-1]^+ - [\fx(t) - 1]^+ + \fe(t) - \fr(t); \\
\label{eq-fx}
\fx(t) & = & \fx(0) + \fe(t) - \int_0^t \lan h^s, \fmeas_s \ran
\,ds -
\fr(t); \\
\label{eq-fr}
\fr(t) & = & \int_0^t \biggl( \int_0^{[\fx(s) - 1]^+} h^r (
( F^{\freneg_s} )^{-1}
(y) ) \,dy \biggr) \,ds;\\[-15pt]\nonumber
\end{eqnarray}
%
%
\begin{eqnarray}
\label{eq-fnonidling}
1 - \lan{\mathbf1}, \fmeas_t \ran&=& [1 - \fx(t)]^+;
\\
%
%
\label{fq-eta} [\fx(t)-1]^+&\leq&\lan{\mathbf1},
\freneg_t \ran
.
\end{eqnarray}
\end{defn}

Note that these fluid equations are not of the same
form as those given in Definition 3.3 of \cite{kanram08b} because the analogs
of (\ref{eqnnu}) and (\ref{eqneta}) are presented in dynamical form
in \cite{kanram08b} and are only required to be satisfied for
continuous functions with compact support (in particular, see equations
(3.9) and (3.11) of \cite{kanram08b}).
However, these two pairs of equations are equivalent
due to Theorem~4.1 of \cite{kasram07} or, equivalently,
Proposition 4.1 of \cite{kanram08b}, and can be shown to
hold for the larger class bounded measurable functions using standard
monotone convergence arguments. Theorems 3.5 and 3.6 of \cite
{kanram08b} show that under some mild assumptions on the
input data $\overline{E}$, $\overline{\nu}_0$ and $\overline{\eta}_0$ and the
hazard rate functions $h^r$ and $h^s$ (which are
stated as Assumptions \ref{ass-en} and
\ref{ass-h} here), there exists
a unique solution to the fluid equations.

For future purposes, note that if $(\fx, \fmeas, \freneg)$ satisfy
the fluid equations for some $(\fe,\fx(0),\fmeas_0, \freneg_0)
\in\newspace$, then $\fk$
also satisfies
%
%
\begin{equation}
\label{eq-fk2}
\fk(t) = \lan{\mathbf1}, \fmeas_t \ran- \lan{\mathbf1}, \fmeas
_0\ran
+ \int_0^t \lan h^s, \fmeas_s \ran \,ds.
\end{equation}
Indeed, this is simply the mass balance equation for the
fluid in service and can be derived from (\ref{eq-fk}), (\ref{eq-fx}) and
(\ref{eq-fnonidling}).
Moreover, combining (\ref{eq-fk2}) and~(\ref{eqnnu}), with $f = {\mathbf1}$, and using an
integration-by-parts argument (see Corollary~4.2 of \cite{kanram08b}),
it is easy to see that
$\fk$ satisfies the renewal equation
%
%
\begin{eqnarray}
\label{disfkl}
\fk(t) & = & \lan{\mathbf1}, \fmeas_t \ran- \lan{\mathbf1},
\fmeas_0 \ran+\int_{[0,H^s )}
\frac{G^s(x+t)- G^s(x)}{1 - G^s(x)} \fmeas_0 (dx) \nonumber\\[-8pt]\\[-8pt]
& &{} + \int_0^t
g^s(t-s)\fk(s) \,ds.\nonumber
\end{eqnarray}
Since the first two terms on the right-hand side are bounded,
by the key renewal theorem (see, e.g., Theorem 4.3 in Chapter V of \cite
{asmbook}),
$\fk$ admits the representation
%
%
\begin{eqnarray}
\label{renewfk1}
\fk(t) &=& \lan{\mathbf1}, \fmeas_{t} \ran- \lan{\mathbf1}, \fmeas
_0 \ran+\int_{[0,H^s )}
\frac{G^s(x+t)- G^s(x)}{1 - G^s(x)} \fmeas_0 (dx) \nonumber\\
&&{} + \int_0^t \biggl( \lan{\mathbf1}, \fmeas_{t-s} \ran- \lan
{\mathbf1}, \fmeas_0 \ran\\
&&\hspace*{30.8pt}{}+\int_{[0,H^s )}
\frac{G^s(x+t-s)- G^s(x)}{1 - G^s(x)} \fmeas_0 (dx) \biggr)
u^s(s) \,ds,
\nonumber
\end{eqnarray}
where $u^s$ is the density of the renewal function $U^s$ associated with
$G^s$
($u^s$ exists because $G^s$ is assumed to have a density).
Also, it will prove convenient to introduce the fluid
queue length process $\fq$ defined by
%
%
\begin{equation}
\label{eq-fq}
\fq(t) \doteq[\fx(t) - 1]^+,\qquad t \in[0,\infty).
\end{equation}
For every $t\in[0,\infty)$, the inequality in (\ref{fq-eta}) implies
%
%
\begin{equation}
\label{eq-fqfreneg}
\fq(t) \leq\lan{\mathbf1}, \freneg_t \ran,
\end{equation}
and
(\ref{eq-fk}) and (\ref{eq-fq}), when combined,
show that
%
%
\begin{equation}\label{qt-conserve}
\fq(0)+\fe(t)=\fq(t)+\fk(t)+\fr(t).
\end{equation}

The fluid equations without abandonment
can be defined in a similar fashion.
Let
%
%
\begin{eqnarray}
\label{def-newspace2}
\tilde\newspace&\doteq&\{ (e,x,\nu) \in\incspace\times\R_+
\times\mmfs\dvtx\nonumber\\[-8pt]\\[-8pt]
&&\hspace*{102.8pt}1 - \lan{\mathbf1}, \nu\ran= [1-x]^+ \}.\nonumber
\end{eqnarray}
\begin{defn} \label{defwa}
Given any $(\fe, \fx(0), \fmeas_0) \in\tilde\newspace$,
we say $(\fx, \fmeas) \in\R_+ \times{\cal M}_F[0, H^s)$
is a solution to the associated fluid equations in the absence of
abandonment if for every $t \in[0,\infty)$,
the second inequality in (\ref{cond-radon}) holds, and equations
(\ref{eqnnu}), (\ref{eq-fk}), (\ref{eq-fx}) and (\ref{eq-fnonidling})
hold with
$\fr\equiv0$.
\end{defn}
%
%
\begin{remark}
The case when customers do not renege corresponds to the case when
the patience time distribution $G^r$ has unit mass at $\infty$.
Formally setting $dG^r = \delta_{\infty}$ in Definition \ref{def-fleqns},
we obtain the fluid limit
equations in the absence of abandonment specified in Definition \ref
{defwa} (also refer to Definition 3.3 in \cite{kasram07}). In fact,
in this case $G^r(x)=0$ and hence, $h^r(x)=0$ for all $x\in[0,\infty
)$. From this and (\ref{eq-fr}) we see that $\fr(t)=0$ for all $t\geq
0$. Also, note that (\ref{cond-radon}), (\ref{eqnnu}), (\ref{eq-fx}),
(\ref{eq-fnonidling}) and (\ref{eq-fk2}) are equivalent to
(3.4)--(3.8) of Definition 3.3 in \cite{kasram07}. At last, by letting
$f={\mathbf1}$ in (\ref{eqneta}), since $G^r$ is zero on $[0,\infty
)$, we have $\lan{\mathbf1}, \freneg_t \ran= \lan{\mathbf1},
\freneg_0 \ran+ \fe(t)$. On the other hand, by (\ref{eq-fx}) and~(\ref{def-newspace}), we have
\[
[\fx(t)-1]^+ \leq\bigl[[\fx(0)-1]^+ + \fe(t)\bigr]^+ \leq
[\lan{\mathbf1}, \freneg_0 \ran+ \fe(t)]^+=\lan
{\mathbf1}, \freneg_t \ran.
\]
This shows that (\ref{fq-eta}) holds automatically when there is no
abandonment.
\end{remark}

\subsection{Invariant manifold} \label{secIM}
We now introduce a set of states associated with the fluid equations
described in Definition \ref{def-fleqns}, which\vadjust{\goodbreak} we call the \textit
{invariant manifold}. As shown in Section \ref{secconv}, when the
invariant manifold consists of a single point, it is the limit of the
scaled sequence of convergent stationary distributions $(\fxns
,\fmeasns,\frenegns)=\frac{1}{N}(\xns,\measns,\renegns)$.
\begin{defn}[(Invariant manifold)]
\label{def-IS}
Given $\lambda\in(0,\infty)$, a state
$(x_0,\meas_0,\allowbreak\reneg_0)\in\R_+ \times\mmfs\times\mmfr$ such
that $(\lambda{\mathbf1}, x_0, \meas_0,\reneg_0)\in\newspace$ is
said to
be invariant for the fluid equations described in Definition \ref
{def-fleqns} with arrival rate $\lambda$
if the solution $(\fx,\fmeas,\freneg)$ to the fluid equations
associated with
the input data $(\lambda{\mathbf1}, x_0, \meas_0,\reneg_0)$
satisfies $(\fx(t),\fmeas_t,\freneg_t)= (x_0,\meas_0,\reneg_0)$
for all $t
\geq0$. The set of all invariant states for
the fluid equation with rate $\lambda$ will be referred to as
the \textit{invariant manifold} (associated with the fluid equations with
rate $\lambda$).
\end{defn}
\begin{theorem}[(Characterization of the invariant manifold)]
\label{th-invman}
Given $\lambda\in(0,\infty)$, the set ${\cal I}_\lambda$ defined in
$(\ref{eq-invman})$ is the invariant manifold
associated with the fluid equations with arrival rate $\lambda$.
\end{theorem}

Theorem \ref{th-invman} is a consequence of the next two lemmas.
Let $\lambda\in(0,\infty)$ and $(x_0,\meas_0,\reneg_0)$ be an
invariant state according to Definition \ref{def-IS}. Then the unique
solution $(\fx,\fmeas,\freneg)$ to the fluid equations associated
with the input data $(\lambda{\mathbf1}, x_0,\nu_0,\eta_0)\in
\newspace$ satisfies $(\fx(t),\fmeas_t,\freneg_t)= (x_0,\meas
_0,\reneg_0)$ for all $t \geq0$. Let~$\fq$, $\fr$, $\fk$ be the
associated auxiliary processes satisfying (\ref{eq-fq}), (\ref{eq-fr}),
(\ref{eq-fk}), and recall the definition of the measures
$\nu_*$ and $\eta_*$ given in (\ref{def-invmeas}) and (\ref
{def-invrenegs}), respectively.
\begin{lemma} \label{lem-reneg}
If $(x_0,\meas_0,\reneg_0)$ is an invariant state, then
$\reneg_0(dx)=\lambda(1-G^r(x))\,dx=\lambda\eta_*(dx)$.
\end{lemma}
\begin{pf}
On substituting the relation $\reneg_t = \reneg_0, t\geq0$, into
(\ref{eqneta}), we see that for every $f \in\cal C_b(\R_+)$ and
$t\in[0,\infty)$,
%
%
\begin{eqnarray}
&&\int_{[0,H^r)} f (x) \reneg_0 (dx) \nonumber\\
&&\qquad= \int_{[0,H^r)} f(x+t)
\frac{1 - G^r(x+t)}{1 - G^r(x)} \reneg_0 (dx)\\
&&\qquad\quad{} + \lambda\int_0^t
f(s) \bigl(1 - G^r(s)\bigr) \,ds.
\nonumber
\end{eqnarray}
Sending $t\rightarrow\infty$ and
applying the dominated convergence theorem, the first term vanishes and
we obtain
\[
\int_{[0,H^r)} f (x) \reneg_0 (dx)=\lambda\int_0^\infty f(s) \bigl(1
- G^r(s)\bigr) \,ds=\int_{[0,H^r)} f(s) \lambda\bigl(1 - G^r(s)\bigr) \,ds.
\]
It then follows that $\reneg_0(dx)=\lambda\eta_*(dx)$.
\end{pf}
%
%
\begin{lemma} \label{lem-meas} If $(x_0,\meas_0,\reneg_0)$ is an invariant
state, then
$\meas_0(dx)=(\lambda\wedge1)\nu_*(dx)$, $x_0=\lambda$ if $\lambda<1$
and $x_0\in B_\lambda$ if $\lambda\geq1$. Moreover,
if either $x_0 = \lambda< 1$, or
$\lambda> 1$ and $x_0 \in B_{\lambda}$, then $(x_0,(\lambda\wedge
1)\nu_*, \lambda\reneg_*)$ is an invariant state.
\end{lemma}
\begin{pf}
Suppose $(x_0,\meas_0,\reneg_0)$ is an invariant state.
Since $\fx(t)=x_0$, we have $\fq(t)=\fq(0)$ by (\ref{eq-fq}).
Since, in
addition, $\freneg_t=\reneg_0=\lambda\eta_*$
by Lemma \ref{lem-reneg}, we have
\[
\int_0^{[\fx(t) - 1]^+} h^r (( F^{\freneg_t} )^{-1}
(y)) \,dy= \int_0^{[x_0 - 1]^+} h^r (( F^{\lambda
\invrenegs})^{-1}(y)) \,dy.
\]
Let $p$ denote the term on the right-hand side of the above display.
Then for each \mbox{$t\geq0$}, by (\ref{eq-fr}) we have $\fr(t)=pt$ and by
(\ref{qt-conserve}) we have $\fk(t)=(\lambda-p)t$. Substituting
$\fmeas_t = \meas_0$ in (\ref{eqnnu}), we obtain for every $f \in
\cal C_b(\R_+)$ and $t\in[0,\infty)$,
%
%
\begin{eqnarray}
&&\int
_{[0,H^s)} f (x) \meas_0 (dx) \nonumber\\
&&\qquad= \int_{[0,H^s)} f(x+t)
\frac{1 - G^s(x+t)}{1 - G^s(x)} \meas_0 (dx)
\\
&&\qquad\quad{}+ \int_0^t f(s) \bigl(1 -
G^s(s)\bigr)(\lambda-p) \,ds.\nonumber
\end{eqnarray}
Sending $t\rightarrow\infty$ and applying
the dominated convergence theorem, we obtain
\begin{eqnarray*}
\int_{[0,H^s)} f (x) \meas_0 (dx)
&=&(\lambda-p) \int_0^\infty f(s) \bigl(1 -
G^s(s)\bigr) \,ds\\
&=&(\lambda-p) \int_{[0,H^r)} f(s) \bigl(1 - G^s(s)\bigr) \,ds.
\end{eqnarray*}
Thus,
$\meas_0(dx)=(\lambda-p) \nu_*(dx)$ and hence, $\lan{\mathbf1},
\meas_0 \ran= \lambda
- p$.

To show that $\meas_0(dx)=(\lambda\wedge1)\nu_*(dx)$, it suffices
to show
that $\lambda-p= \lan{\mathbf1}, \meas_0 \ran= \lambda\wedge1$.
If $x_0\leq1$,
then $p=0$ by its definition. Hence, $\meas_0(dx)=\lambda\nu_*(dx)$ and
$\lambda= \lan{\mathbf1},\meas_0 \ran\leq1$. Thus, in this case,
$\lambda-p=\lambda\wedge1$. On the other hand, if $x_0> 1$, it
follows from
(\ref{eq-fnonidling}) that $\lan{\mathbf1},\meas_0 \ran= 1$. Since
we also have
$\lan{\mathbf1},\meas_0 \ran= \lambda-p$, it follows that $\lambda
= p + 1 \geq
1$. Thus, in
this case too, we have $\lambda-p=\lambda\wedge1$. This proves the
first assertion of the lemma.

For the second assertion of the lemma, we observe that when $\lambda
<1$, the equality $\lambda-p=\lambda\wedge1$ implies $p=0$ and $\lan
{\mathbf1},\meas_0 \ran= \lambda< 1$. Hence, (\ref{fnonidling})
implies $x_0=\lan{\mathbf1},\meas_0 \ran= \lambda$. If $\lambda
\geq1$, we have $\meas_0(dx)= \nu_*(dx)$ and the equality $\lambda
-p=\lambda\wedge1$ implies $p=\lambda-1$. Then $x_0\geq\lan
{\mathbf1}, \meas_0 \ran=1$ and
\[
\lambda G^r\bigl((F^{\lambda
\invrenegs})^{-1}\bigl((x_0-1)^+\bigr)\bigr)=\int
_0^{(x_0-1)^+}h^r((F^{\lambda
\invrenegs})^{-1}(y))\,dy=p=\lambda-1.
\]
Hence, $x_0$ belongs to the set
$B_\lambda$ defined in (\ref{eq-fp}).
The last assertion can be
verified
directly by substituting the initial condition into the fluid equations.
This completes the proof of the lemma.\vadjust{\goodbreak}
\end{pf}

\section{The limit of scaled stationary distributions}
\label{secconv}

This section is devoted to the proof of Theorem
\ref{thm-convstat}.
Suppose Assumptions\vspace*{1pt}
\ref{as-mean} and \ref{assinterdis} hold and let
$\fyns=(\overline{\alpha}_{E,*}^{(N)},\fxns,\fmeasns,\frenegns)$,
$N\in\N$, be a sequence of scaled stationary distributions for the $N$-server
queue, which exists by Theorem \ref{thm-SD}.
When Assumption \ref{ass-unique} also holds,
let
$(x_*, (\lambda\wedge1)\nu_*,\lambda\eta_*)$ be the unique element
of the invariant manifold ${\cal I}_\lambda$. The main result of this
section is to show that, as $N\ra\infty$,
%
%
\begin{equation}\label{Sdconv} \bigl(\fxns
,\fmeasns,\frenegns\bigr) \Rightarrow\bigl(x_*, (\lambda\wedge1)\nu
_*,\lambda\eta_*\bigr).
\end{equation}

We first show in Section \ref{sectight} that
the sequence $\{(\fxns,\fmeasns,\frenegns), N\in\N\}$ is tight.
Then, in
Section \ref{secsubconv}, we show that (without imposing
Assumption~\ref{ass-unique}) the weak
limit of every convergent subsequence must almost surely be an
invariant state.
When there is a unique invariant state, this proves~(\ref{Sdconv}).
Note that the method of proof
does not explicitly require that the stationary distribution
for each $N$-server queue be unique. For each
$N \in\N$, recall the definition given in (\ref{fl-scaling})
of the fluid-scaled state process
%
%
\begin{equation}
\label{disbarY}
\fyn= \bigl(\overline{\alpha}{}^{(N)}_E,\fxn,\fmeasn,\frenegn\bigr)
\end{equation}
for the
$N$-server queue with abandonment associated with the initial condition
$\fyn(0)=(\overline{\alpha}{}^{(N)}_{E,*},\fxns,\fmeasns,\frenegns)$.
Let $\fqn, \frn, \fkn$ be the fluid-scaled auxiliary processes
associated with $\fyn$ that were introduced in Section \ref{secrepdyn}.

\subsection{Tightness} \label{sectight} To establish tightness of
the sequence
$\{\fyns\}_{N \in\N}$, we will make
use of the criteria for tightness of measure-valued random variables
given in Proposition \ref{proptight}.
\begin{lemma}
Let $c\in[0,H^r)$. Then, for each integer $n\geq2$,
$\frenegns$ and $\fmeasns$ satisfy the following relations:
%
%
\begin{eqnarray}\label{dis-stateta}
&&\E\bigl[\frenegns[c,H^r)
\bigr]\nonumber\\[-2pt]
&&\qquad=\E\biggl[\int_{[0,H^r)}\frac
{1-G^r(x+nc)}{1-G^r(x)}\frenegns(dx)\biggr] \\[-2pt]
&&\qquad\quad{}+ \E\Biggl[\int_{[0,c]}\sum_{j=2}^n\bigl(1-G^r(jc-s)\bigr)\,d\fen(s)\Biggr],
\nonumber\\[-2pt]
%
\label{dis-statnu}
&&\E\bigl[\fmeasns[c,H^s)
\bigr]\nonumber\\[-2pt]
&&\qquad=\E\biggl[\int_{[0,H^s)}\frac
{1-G^s(x+nc)}{1-G^s(x)}\fmeasns(dx)\biggr] \\[-2pt]
&&\qquad\quad{}+ \E\biggl[\int_{[0,c]}\sum_{j=2}^n\bigl(1-G^s(jc-s)\bigr)\,d\fkn(s)\biggr].
\nonumber\vadjust{\goodbreak}
\end{eqnarray}
\end{lemma}
\begin{pf}
We only prove (\ref{dis-stateta}) because (\ref{dis-statnu}) can be
proved in the same~way. Fix $c\in[0,H^r)$. Dividing both sides of
(\ref{dis-key1}) by $N$ and setting $\frenegn_0=\frenegns$,
we obtain for each bounded measurable function $f$ on $\R_+$ and $t>0$,
%
%
\begin{eqnarray} \label{dis-key} \E\bigl[\bigl\lan f, \frenegn_t \bigr\ran
\bigr] & = & \E\biggl[\int_{[0,H^r)}f(x+t)\frac
{1-G^r(x+t)}{1-G^r(x)}\frenegns(dx)\biggr]
\nonumber\\[-8pt]\\[-8pt]
& &{} + \E\biggl[\int_{[0,t]}f(t-s)\bigl(1-G^r(t-s)\bigr)\,d\fen(s)\biggr].
\nonumber
\end{eqnarray}
Since the initial conditions are stationary, $\frenegn_t$ has the same
distribution as $\frenegns$ for every $t\geq0$. Therefore, by substituting
$f=\indd_{[c,H^r)}$ and $t=c$ in~(\ref{dis-key}), and noting that
$F^{(N)}(0)=0$, $\indd_{[c,H^r)}(x+c)=1$ for every $x\geq0$ and
$\indd_{[c,H^r)}(c-s)=0$ for every $s\in(0,c]$, we obtain
\begin{eqnarray*}
\E\bigl[\frenegns[c,H^r)\bigr]&=&\E\bigl[\frenegn
_c[c,H^r)\bigr] \\
&=&\E\biggl[\int_{[0,H^r)}\frac
{1-G^r(x+c)}{1-G^r(x)}\frenegns(dx)\biggr]\\
&=&\E\biggl[\int_{[0,H^r)}\frac{1-G^r(x+c)}{1-G^r(x)}\frenegn_c(dx)\biggr].
\end{eqnarray*}
Next, choosing $f=(1-G^r(\cdot+c))/(1-G^r(\cdot))$ and $t=c$ in (\ref
{dis-key}), we obtain
\begin{eqnarray*} \E\biggl[\int_{[0,H^r)}\frac
{1-G^r(x+c)}{1-G^r(x)}\frenegn_c(dx)\biggr] &=&\E\biggl[\int
_{[0,H^r)}\frac{1-G^r(x+2c)}{1-G^r(x)}\frenegns(dx)\biggr]\\
& &
{}+ \E\biggl[\int_{[0,c]}\bigl(1-G^r(2c-s)\bigr)\,d\fen(s)\biggr].
\end{eqnarray*}
Combining the last two displays, we see that
\begin{eqnarray*} \E\bigl[\frenegns[c,H^r)\bigr] &=&\E\biggl[\int
_{[0,H^r)}\frac{1-G^r(x+2c)}{1-G^r(x)}\frenegns(dx)\biggr]\\
& &
{}+ \E\biggl[\int_{[0,c]}\bigl(1-G^r(2c-s)\bigr)\,d\fen(s)\biggr].
\end{eqnarray*}
Thus, we have shown that (\ref{dis-stateta}) holds for $n=2$. Suppose
that for some integer $m\geq2$, (\ref{dis-stateta}) holds for $n=m$,
that is,
%
%
\begin{eqnarray} \label{dis-ind} \E\bigl[\frenegns[c,H^r)\bigr] &=&\E
\biggl[\int_{[0,H^r)}\frac{1-G^r(x+mc)}{1-G^r(x)}\frenegns(dx)\biggr]\nonumber\\[-8pt]\\[-8pt]
& &
{}+ \E\biggl[\int_{[0,c]}\sum_{j=2}^m\bigl(1-G^r(jc-s)\bigr)\,d\fen(s)\biggr].
\nonumber
\end{eqnarray}
Substituting $f=(1-G^r(\cdot+mc))/(1-G^r(\cdot))$ and $t=c$ in (\ref
{dis-key}) and using the fact that $\frenegn_c$ has the same
distribution as $\frenegns$, we obtain
%
%
\begin{eqnarray}
&&\E\biggl[\int_{[0,H^r)}\frac{1-G^r(x+mc)}{1-G^r(x)}\frenegns
(dx)\biggr] \nonumber\\
&&\qquad= \E\biggl[\int_{[0,H^r)}\frac
{1-G^r(x+mc)}{1-G^r(x)}\frenegn_c(dx)\biggr] \nonumber\\[-8pt]\\[-8pt]
&&\qquad=\E\biggl[\int_{[0,H^r)}\frac{1-G^r(x+(m+1)c)}{1-G^r(x)}\frenegns(dx)\biggr]
\nonumber\\
&&\qquad\quad{}+ \E\biggl[\int_{[0,c]}\bigl(1-G^r\bigl((m+1)c-s\bigr)\bigr)\,d\fen(s)\biggr].\nonumber
\end{eqnarray}
This, together with (\ref{dis-ind}), yields (\ref{dis-stateta}) with
$n=m+1$. This completes the induction argument and we have the desired
result.
\end{pf}
\begin{theorem} \label{thm-tight}
If Assumptions \ref{as-mean} and \ref{assinterdis} are satisfied and
$\overline{\lambda}{}^{(N)} \ra\lambda\in(0,\infty)$, then
the sequence $\{(\fxns,\fmeasns,\frenegns)\}_{N\in\N}$ is tight.
Moreover,
%
%
\begin{equation}\label{disSrenegn}
\sup_{N\in\N}\E\bigl[\bigl\lan
{\mathbf1},\frenegns\bigr\ran\bigr] <\infty.
\end{equation}
\end{theorem}
\begin{pf}
We first show that $\{\frenegns\}_{N\in\N}$ is tight.
Note that $\lan{\mathbf1}, \frenegns\ran$ can be viewed as the
fluid scaled queue
length\vspace*{1pt} process associated with an infinite-server queue with arrival process
$\fen$ and service distribution function $G^r$. By Little's law (cf.
Theorem 2
of \cite{little}), we know that $\E[\lan{\mathbf1}, \frenegns\ran
]=\overline
\lambda{}^{(N)}\theta^r$, where $\theta^r$, the mean of $G^r$, is finite
by Assumption \ref{as-mean}. Due to the assumed convergence of
$\overline\lambda{}^{(N)}$ to $\lambda$, this implies (\ref{disSrenegn}).

Next, note that for each $n$, the function $(1-G^r(\cdot
+nc))/(1-G^r(\cdot))$ is
bounded by $1$ and converges to $0$ as $n\rightarrow\infty$.
Therefore, it follows from the dominated convergence theorem that
%
%
\begin{equation}\label{dis-01}\lim_{n\ra
\infty}\E\biggl[\int_{[0,H^r)}\frac{1-G^r(x+nc)}{1-G^r(x)}
\frenegns(dx)\biggr]=0.
\end{equation}
Sending $n\ra\infty$ on the right-hand side of (\ref{dis-stateta}),
and using
(\ref{dis-01}) and the monotone convergence theorem, we have
%
%
\begin{equation}\label{dis-02}\E\bigl[\frenegns[c,H^r)\bigr]=\E\Biggl[\int
_{[0,c]}\sum
_{j=2}^\infty\bigl(1-G^r(jc-s)\bigr)
\,d\fen(s)\Biggr].
\end{equation}
On the other hand, we also have the simple estimate
%
%
\begin{eqnarray}
\label{dis-ser1}
\E\biggl[\int_{[0,c]}\bigl(1-G^r(2c-s)\bigr) \,d\fen(s)\biggr]
&\leq&\bigl(1-G^r(c)\bigr)\E\bigl[\fen(c)\bigr]\nonumber\\[-8pt]\\[-8pt]
&=& c\bigl(1-G^r(c)\bigr)\frac{\E[\fen(c)]}{c}.
\nonumber
\end{eqnarray}
Carrying out integration-by-parts on $\int_0^\infty(1-G^r(x))\,dx$,
it follows that
\[
\int_{[0,H^r)}\bigl(1-G^r(x)\bigr) \,dx = \lim_{x\ra H^r}
x\bigl(1-G^r(x)\bigr)+\int_{[0,H^r)}xg^r(x) \,dx.
\]
However, since the mean $\theta^r$ is
finite by (\ref{def-mean1}), it follows that $c(1-G^r(c))\rightarrow
0$ as
$c\ra H^r$. In addition, because the elementary renewal theorem implies
that $\E[\fen(c)]/c \rightarrow\overline
\lambda{}^{(N)}$ as $c\ra\infty$ and
$\overline\lambda{}^{(N)}\ra\lambda$ as $N\ra\infty$, it follows
that
%
%
\begin{equation}\label{unif-bound} \limsup_{c\ra
H^r}\sup_N\frac{\E[\fen(c)]}{c} < \infty.
\end{equation}
Thus, taking first the supremum over $N$ and then the limit as $c \ra
H^r$ in~(\ref{dis-ser1}), we obtain
%
%
\begin{equation}\label{dis-1}\lim_{c\ra H^r}\sup_N\E\biggl[\int
_{[0,c]}\bigl(1-G^r(2c-s)\bigr) \,d\fen(s)\biggr] = 0.
\end{equation}
Since $1-G^r(\cdot)$ is a decreasing function, for $s\in[0,c]$,
\[
\sum_{j=3}^\infty c\bigl(1-G^r(jc-s)\bigr)\leq\int_{[2c-s,H^r)}\bigl(1-G^r(x)\bigr) \,dx
\leq
\int_{[c,H^r)} \bigl(1-G^r(x)\bigr) \,dx.
\]
Therefore, we have
\begin{eqnarray*}
&&\sup_N\E\Biggl[\int_{[0,c]}\sum_{j=3}^\infty\bigl(1-G^r(jc-s)\bigr) \,d\fen
(s)\Biggr] \\
&&\qquad\leq
\sup_N\frac{\E[\fen(c)]}{c}\int_{[c,H^r)} \bigl(1-G^r(x)\bigr)
\,dx,
\end{eqnarray*}
which tends to zero as $c\ra H^r$
because of (\ref{unif-bound}) and Assumption \ref{as-mean}.
Combining the last assertion with (\ref{dis-02}) and (\ref{dis-1}),
we see that
%
%
\begin{equation}\label{disser2}\lim_{c\ra H^r}\sup_N \E
\bigl[\frenegns[c,H^r)\bigr] = 0,
\end{equation}
which establishes the second
criterion for tightness. Thus, the sequence $\{\frenegns\}_{N\in\N}$
is tight.

We next\vspace*{1pt} show that $\{\fmeasns\}_{N\in\N}$ is tight. The analog of
(\ref{disSrenegn}) holds for $\{\fmeasns\}_{N\in\N}$ automatically
because $\lan{\mathbf1}, \fmeasns\ran\leq1$ for each $N$. On the
other hand, the analog of (\ref{disser2}) can be shown to hold for $\{
\fmeasns\}_{N\in\N}$ by using~(\ref{dis-statnu}) and an argument
similar to that used above to establish (\ref{disser2}), along with
the additional observation that $\E[\fkn(c)]\leq\E[\fen(c)]+\E
[\lan{\mathbf1}, \frenegns\ran]$ implies $\limsup_{c\ra H^s}\sup
_N\E[\fkn(c)]/ c<\infty$. Thus, the sequence $\{\fmeasns\}_{N\in\N
}$ is also tight.

Finally, we show that the sequence of $\R_+$-valued random variab-\break les~$\{\fxns\}_{N\in\N}$~is~tight. Since $\fxns\leq1+\lan{\mathbf1},
\frenegns\ran$ for each $N$, $\sup_N\E[\fxns]\leq1+ \sup_N\E
[\lan{\mathbf1}$, $\frenegns\ran]$, which is finite due to (\ref
{disSrenegn}). The tightness of $\{\fxns\}_{N\in\N}$ is a direct
consequence of Markov's inequality.
\end{pf}

\subsection{The limit of the stationary distributions} \label{secsubconv}

We now present the proof of our main result.
\begin{pf*}{Proof of Theorem \ref{thm-convstat}}
For each\vspace*{-2pt} $N \in\N$, let $\fyns=
(\overline{\alpha}{}^{(N)}_{E,*},\fxns,\fmeasns$, $\frenegns)$ be
a fluid scaled stationary distribution for the $N$-server
system.\vspace*{1pt}
We will invoke the fluid limit theorem established in Theorem 3.6 of
\cite{kanram08b}
to establish the result. For each $N \in\N$,
let $\overline{Z}{}^{(N)} = (\fxn,\fmeasn,\frenegn)$
be the (fluid scaled) state process for the $N$-server queue with initial
data $(\fen_*, \fxns, \fmeasns,
\frenegns)$. Since the hazard rate functions $h^s$ and $h^r$ satisfy
Assumption \ref{ass-h} (which corresponds to Assumption 3.3 of \cite
{kanram08b}), it follows from
Remark 3.2 and Theorem 3.6 of \cite{kanram08b} that if (a) the
sequence (or subsequence) of initial data
$(\fen_*, \fxns, \fmeasns,
\frenegns)$ converges
in distribution to some random element
$(\fe_*, \tilde{X}_*, \tilde{\nu}_*, \tilde{\eta}_*)$ in the
sense of
Assumption 3.1 of~\cite{kanram08b}, (b) $\fe_*$ is continuous and (c)
$\tilde{\eta}_*$ is a continuous
distribution, then the sequence (subsequence)
$\overline{Z}{}^{(N)}$ converges to a process $\overline{Z}
= (\fx, \fmeas, \freneg)$ that is
the unique solution to the fluid equations
with initial data $(\fe_*, \tilde{X}_*, \tilde{\nu}_*, \tilde{\eta}_*)$.
However, by stationarity for each $N \in\N$ and $t > 0$,
$\overline{Z}{}^{(N)}(t)$ has the same distribution as $\overline{Z}{}^{(N)}(0)$.
This implies that $\overline{Z}$ is the constant process that is
identically equal to the initial data $(\tilde{X}_*, \tilde{\nu}_*,
\tilde{\eta}_*)$,
which in turn implies that $(\tilde{X}_*, \tilde{\nu}_*, \tilde
{\eta}_*)$
is an invariant state for the fluid limit.

Thus, to establish the theorem, it only remains to verify properties (a)--(c)
above. Since Assumptions \ref{as-mean}, \ref{ass-en}(1) and \ref
{assinterdis} hold, by the tightness
result established in Theorem \ref{thm-tight},
it follows that the sequence of stationary ``initial conditions'' $\{
\overline{Y}{}^{(N)}_*\}_{N \in\N}$ is tight.
On the other hand, by basic properties of renewal processes and the assumption
that $\flam^{(N)} \ra\lambda$,
the sequence of scaled stationary arrival processes
$\{\overline E{}^{(N)}_*\}_{N \in\N}$ satisfies
$\fen_* \Rightarrow\overline{E}_*$ as $N \ra\infty$,
where $\overline{E}_* (t) = \lambda t$ for
$t \in[0,\infty)$.
Therefore, there exists a convergent subsequence, which by some abuse
of notation
we denote again by $\{\overline{Y}{}^{(N)}_*\}_{N \in\N}$,
that converges weakly to some limit $\overline{Y}_*$ of the form
$\overline{Y}_* = (\lambda{\mathbf1}, \tilde{X}_*, \tilde{\nu}_*,
\tilde{\eta}_*)$.
This immediately shows that properties (a) and (b) above are satisfied.
It only remains to show that $\tilde{\eta}_*$ has a continuous distribution.
Now, by the proof of Theorem 7.1 of \cite{kanram08b} (note that
Assumption 3.2
of \cite{kanram08b} is not used for this part of the proof) it follows that
the inequality (3.39) of \cite{kanram08b} holds and that
$\freneg$ satisfies the dynamical equation (3.42) of \cite{kanram08b}, with
$\freneg_0 = \tilde{\eta}_*$ and~$\fe$ replaced by
$\fe_*$.
By Theorem 4.1 of \cite{kasram07} (equivalently,
Proposition 4.1 of \cite{kanram08b}), it follows that
$\freneg$ satisfies the fluid equation (\ref{eqneta})
with $\freneg_0 = \tilde{\eta}_*$. In particular, also using
the fact that $\fe_*(t) = \lambda t$ and
$\freneg_t$ has the same distribution as $\tilde{\eta}_*$, this
implies that
for every bounded measurable $f$ on $[0,H_r)$ and any $t > 0$,
%
%
\begin{eqnarray}\label{dis-rest}\lan f, \tilde\reneg_* \ran
&\stackrel{(d)}{=}&
\int_{[0,H^r)}f(x+t)\frac{1-G^r(x+t)}{1-G^r(x)}
\tilde\reneg_*(dx)\nonumber\\[-8pt]\\[-8pt]
&&{} + \int_0^tf(t-s)\bigl(1-G^r(t-s)\bigr)\lambda \,ds.\nonumber
\end{eqnarray}
%
Now, sending $t \ra\infty$ on the right-hand side, using
the dominated convergence theorem (which is justified by the
bound $\lan1, \tilde{\eta}_* \ran< \infty$ a.s. established in
Theorem~\ref{thm-tight}), we see that the right-hand side equals $\lambda\lan
f, \eta_*\ran$.
This shows that $\tilde{\eta}_*$ has the same distribution as
$\lambda
\eta_*$, which in particular proves that $\tilde{\eta}_*$ is a continuous
distribution. This
completes the proof of property (c). Thus, we have shown that any
convergence subsequence of the stationary distribution converges to an invariant
state of the fluid limit. When the manifold consists of a single element,
the usual argument by contradiction then shows that
the original sequence of stationary distributions converges
to this point.
\end{pf*}

\section{Concluding remarks}
\label{subs-counteregs}

We can establish ergodicity of the state processes under
an additional condition.
Let
\[
\varrho^r \doteq\sup\{u\in[0,H^r)\dvtx g^r=0 \mbox{ a.e. on } [a,a+u]
\mbox{ for some } a\in[0,\infty)\}
\]
and
\[
\varrho^s \doteq\sup\{u\in[0,H^s)\dvtx g^s=0 \mbox{ a.e. on } [a,a+u]
\mbox{ for some } a\in[0,\infty)\}.
\]

\begin{ass} \label{assPosF}The following three conditions hold:
\begin{longlist}[(1)]
\item[(1)] $H^r=H^s=\infty$;
\item[(2)] $\varrho\doteq\varrho^r\vee\varrho^s<\infty$;
\item[(3)] For every interval $[a,b]\subset[0,\infty)$ with $b-a>0$,
$F^{(N)}(b)-F^{(N)}(a)>0$.
\end{longlist}
\end{ass}
\begin{theorem}
\label{thm-ergodic}
Suppose Assumptions \ref{as-mean}--\ref{assinterdis} and
\ref{assPosF} hold. Then the Markov process $\{Y_t,
{\cal F}_t, t\geq0\}$ is ergodic in the sense that it has
a unique stationary distribution and the
distribution of $Y(t)$ converges in total variation, as $t\ra\infty$,
to this
unique stationary distribution.
\end{theorem}

Theorem \ref{thm-ergodic}, whose proof is deferred to the \hyperref[secPHR]{Appendix}, validates
the rightward arrow at the top of the
``interchange of limits'' diagram presented in
Figure \ref{diag-comm}. On the other hand, the fluid limit theorem
(Theorem 3.6 of~\cite{kanram08b}) justifies the downward\vadjust{\goodbreak}
arrow on the left-hand side of Figure \ref{diag-comm}.
The focus of this work has been on understanding the
convergence represented by the downward arrow on the right-hand side
of Figure \ref{diag-comm}.
%
%
\begin{figure}
$
\begin{CD}
(\fxn(t),\fmeasn_t,\frenegn_t) @>>> (\fxns,\fmeasns,\frenegns) \\
@VVV @VVV \\
(\fx(t),\fmeas_t,\freneg_t) @>?>> (\fx_*,\fmeas_*,\freneg_*)
\end{CD}
$
\caption{Interchange of limits diagram.}
\label{diag-comm}
\end{figure}
When there is a~unique invariant state, this
convergence is established in Theorem \ref{thm-convstat}.
Although this question is not directly relevant to the characterization
of the stationary distributions, it is natural in this
setting to ask whether the diagram in Figure \ref{diag-comm} commutes,
namely, whether the fluid limit from any
initial condition converges as $ t \ra\infty$ to the unique
invariant state.
In Section~\ref{subs-longtime} we briefly discuss why the study of
the long-time behavior of the fluid limit
is a nontrivial task.
Furthermore,
in Section \ref{subs-interchange} we provide a very simple
counterexample that shows that the diagram in
Figure \ref{diag-comm} need not commute and thus,
the limits $N \ra\infty$ and $t \ra\infty$ cannot always be
interchanged.


\subsection{Long-time behavior of the fluid limit}
\label{subs-longtime}

The long-time behavior of the fluid limit is nontrivial even
in the absence of abandonment. For example, in the absence of abandonment,
it was proved in Theorem 3.9 of \cite{kasram07} that $\fmeas_t \ra
\lambda\meas_*$ as $t \ra\infty$ when $\lambda\in[0,1]$, the
service time
distribution $G^s$ has a second moment and its hazard rate function
$h^s$ is either bounded or lower-semicontinuous on $(m_0,H^s)$ for some
$m_0<H^s$. The question of whether
the second moment condition on the distribution is necessary for this
convergence is still unresolved.
Even under the second moment assumption,
the long-time behavior of the component $\fx$ of the fluid limit
is not easy to describe except in the cases when (i)~the system is
subcritical ($\lambda< 1$) or (ii) when the system is critical or
supercritical ($\lambda\geq1$)
and
the service distribution is exponential.
In case (i), it follows from Theorem 3.9
of \cite{kasram07} that $\fx(t) \ra\lambda\lan{\mathbf1},
\meas_*\ran$ as $t \ra\infty$, whereas in case (ii),
if the initial condition satisfies $\fx(0) \geq1$ and $\fmeas_0 \in
{\cal
M}_F[0,\infty)$, then it is easy to see that the fluid limit
is given explicitly by $\fx(t) = \fx(0) + (\lambda- 1) t$ and
$\fmeas_t(dx) =
\indd_{[0,t]}e^{-x}\,dx+\indd_{(t,\infty)}(x)e^{-t}\fmeas_0(d(x-t))$.
Therefore,\vspace*{1pt}
at criticality ($\lambda= 1$), if
$\fx(0) = 1$ then $\fx(t)=\fx(0)$ for every $t>0$. In particular,
$\fx(t) \ra1$ as $t \ra\infty$.
However, as the following example demonstrates,
the critical fluid limit need not converge to $1$ [even if
critically loaded and with initial condition $\overline{X}(0)=1$] when
the service distribution is not exponential.
%
%
\begin{example}
Let the fluid arrival rate be $\overline{E}(t) = t$, $t > 0$,
and let the service time distribution $G^s$ be the Erlang
distribution with density
\[
g^s(x)=4xe^{-2x},\qquad x\geq0.\vadjust{\goodbreak}
\]
A simple
calculation shows that
$\int_0^\infty(1-G^s(x)) \,dx = 1$. Let $(\fx, \fmeas)$ be the
solution to
the fluid equations
without abandonment (see Definition \ref{defwa})
associated with the initial condition $({\mathbf1}, 1, \delta_0)$. We
show below
that in this~case,
$\lim_{t \ra\infty} \fx(t) = 5/4$, which is bigger than $1 = \fx(0)$.
In fact, since
$\fmeas_0=\delta_0$, a~straightforward calculation shows that
\[
\lan h^s, \fmeas_0 \ran= \int_0^\infty
\frac{g^s(x)}{1-G^s(x)}\fmeas_0(dx)=\frac{g^s(0)}{1-G^s(0)}=g^s(0)=0.
\]
Define
\[
\kappa\doteq\inf\{t\geq0\dvtx\lan h^s, \fmeas_t \ran\geq1\}.
\]
The hazard rate function $h^s$ is bounded and continuous and
$\lan h^s, \fmeas_0 \ran< \lambda= 1$.
Therefore, substituting $h^s$ in (\ref{eqnnu}), it is clear that
$\kappa> 0$
and $\lan h^s, \fmeas_t \ran< \lambda= 1$ for
$t \in[0,\kappa)$. In turn, by the nonidling condition, this implies
$\lan{\mathbf1}, \fmeas_t \ran= 1$ and
$d\fk/dt = \lan h^s, \fmeas_t\ran$ and, by (\ref{eqnnu}), for $t
\in
[0,\kappa)$,
\[
\lan h^s, \fmeas_t \ran= g^s(t) + \int_0^t g^s(t-s)\,\frac{d\fk
}{dt}(s) \,ds
= g^s(t) + \int_0^t g^s(t-s)\lan h^s, \fmeas_s \ran \,ds.
\]
Applying the key renewal theorem to the above equation, it
follows that
\[
\lan h^s, \fmeas_t \ran= u^s(t) =
1-e^{-4t}.
\]
Since $u^s(t)<1$ for all $t\geq0$, we must have that $\kappa=
\infty$, $\lan{\mathbf1}, \fmeas_t \ran= 1$ for all $t\geq0$,
and
\[
\lim_{t\rightarrow\infty}\fq(t)=\int_0^\infty\bigl(1-u^s(t)\bigr)\,dt =
\int_0^\infty e^{-4t}\,dt = 1/4,
\]
which yields the convergence of $\fx(t)$ to
$5/4$ as $t\ra\infty$.

To emphasize that this phenomenon is not the consequence of the fact
that the initial
condition was chosen to be singular with respect to Lebesgue measure,
we show
that we can modify the above example by choosing
$\fmeas_0$ to be absolutely continuous with respect to the Lebesgue
measure. For example, for some $\alpha\in(0,\infty)$, define
\[
q(x) \doteq\cases{
\displaystyle \frac{1+2x}{\alpha+\alpha^2}, &\quad
if $x\in[0,\alpha]$, \vspace*{2pt}\cr
0, &\quad otherwise,}
\]
and let $\fmeas_0(dx)=q(x)\,dx$. Then $\lan{\mathbf1}, \fmeas_0 \ran
= \int_0^\alpha
q(x)\,dx = 1$,
$\lan h^s, \fmeas_t \ran=1-((1-\alpha)/(\alpha+1))e^{-4t}$ for each
$t\geq
0$.
Hence, when $\alpha<1$ we have $\lan h^s, \fmeas_t \ran< 1$ and
$\lan{\mathbf1},
\fmeas_t \ran= 1$ for all $t\geq0$. This implies that, when
$\alpha<1$,
\[
\lim_{t\rightarrow\infty}\fq(t)=\int_0^\infty
\frac{1-\alpha}{\alpha+1} e^{-4t}\,dt = \frac{1-\alpha}{4(\alpha+1)}>0,
\]
showing
that $\lim_{t\ra\infty}\fx(t)>1$.
\end{example}

\subsection{A counterexample (invalidity of the interchange of limits)}
\label{subs-interchange}

In this section we show that even for an $M/M/N$ queue (both with and without
abandonments), the ``interchange of limits'' need not hold, that is,
the diagram presented in Figure~\ref{diag-comm} may not commute.

Consider the sequence of state processes $(\xn, \measn)$, $N \in\N
$, of
$N$-server queues without abandonment,
where the service time distribution $G^s$ is exponential with rate $1$. For
the $N$th queue, let the arrival process $\en$ be a~Poisson process with
parameter $\lambda^{(N)}=N-1$ and suppose that there exists $\fmeas_0
\in
{\cal M}_F[0,\infty)$ with $\lan{\mathbf1}, \fmeas_0 \ran\leq1$
such that a.s., as $N \ra\infty$,
%
%
\begin{equation}\label{eg-convint}
\bigl(\fxn(0), \fmeasn_0\bigr) \ra(2, \fmeas_0).
\end{equation}
%
Given the exponentiality of the service time distribution, it immediately
follows that Assumption 2 of \cite{kasram07} is satisfied. Moreover,
because (\ref{eg-convint})
holds and $\overline\lambda^{(N)} = (N-1)/N \ra1$ as $N \ra\infty$,
it follows that Assumption 1 of \cite{kasram07} also holds with
$\lambda=
1$.
On the other hand, since $G^r(x)=0$ for all $x\in[0,\infty)$,
Assumption 2
fails to hold
because in this case $B_1=[1,\infty)$ and so the invariant manifold has
uncountably
many points.

Now, because Assumptions 1 and 2 of \cite{kasram07}
are satisfied, we can apply the fluid limit result
in Theorem 3.7 of \cite{kasram07} to conclude that, almost surely,
as $N \ra\infty$,
$(\fxn, \fmeasn)$
converges weakly to the unique solution $(\fx, \fmeas)$ of
the fluid equations associated with initial data $({\mathbf1}, 2,
\fmeas_0)$, and using the
exponentiality of the service time distribution, it is easily verified
that the fluid limit
is given explicitly by $\fx(t) = \fx(0)=2$ and
$\fmeas_t(dx) = \indd_{[0,t]}e^{-x}\,dx+\indd_{(t,\infty
)}(x)e^{-t}\fmeas_0(d(x-t))$.

For each $N\in\N$, because the arrival rate, which equals $N-1$, is
less than the total service rate $N$, by (3.2.4) and (3.2.5) of \cite
{BDPS} it follows that $\xn$ is ergodic and has the following
stationary distribution:
\[
\PP\bigl(\xns=k\bigr) = \cases{
\displaystyle \frac{(N-1)^k}{k!}p_0, &\quad if $k=0,1,\ldots,N-1$,\vspace*{2pt}\cr
\displaystyle \frac{(N-1)^k}{N!N^{k-N}}p_0, &\quad if $k=N,N+1,\ldots,$}
\]
where
\[
p_0\doteq\Biggl\{\sum_{i=0}^{N-1}\frac{(N-1)^i}{i!}+\frac
{(N-1)^N}{(N-1)!}\Biggr\}^{-1}.
\]
It follows from Stirling's formula that
\begin{eqnarray*}
\lim_{N\ra\infty} \frac{\sum_{i=0}^{N-1}{(N-1)^i}/{i!}}{
{(N-1)^N}/{(N-1)!}} &=& \lim_{N\ra\infty} \frac{\sum_{i=0}^{N-1}
{(N-1)^i}/{i!}}{({1}/{\sqrt{2\pi}})\sqrt{N-1}e^{N-1}}\\
&\leq&\lim
_{N\ra\infty} \frac{\sum_{i=0}^{\infty}{(N-1)^i}/{i!}}{(
{1}/{\sqrt{2\pi}})\sqrt{N-1}e^{N-1}}= 0.
\end{eqnarray*}
For each $\varepsilon>0$, elementary calculations show that
\begin{eqnarray*}
\PP\bigl(\xns\geq N+\varepsilon N\bigr)&=&\sum_{k=N+\varepsilon N}^\infty
\frac
{(N-1)^k}{N!N^{k-N}}p_0 \\ &=& \frac{N^N}{N!}p_0 \sum
_{k=N+\varepsilon N}^\infty\biggl(\frac{N-1}{N}\biggr)^k \\ &=&
\frac{N^N}{N!}p_0 \biggl(\frac{N-1}{N}\biggr)^{N+\varepsilon N}N \\
&=& \frac{(N-1)^N}{(N-1)!}p_0\biggl(\frac{N-1}{N}\biggr)^{
\varepsilon N}
\end{eqnarray*}
and
\[
\PP\bigl(\xns\leq N-\varepsilon N\bigr)=\sum_{k=0}^{N-\varepsilon N} \frac
{(N-1)^k}{k!}p_0.
\]
Combining the above three displays, we then have for each $\varepsilon>0$
%
%
\begin{equation} \label{mmnb}
\lim_{N\ra\infty}\PP\bigl(\fxns\geq1+\varepsilon\bigr) = \lim_{N\ra
\infty}\PP\bigl(\xn_*\leq1-\varepsilon\bigr)=0.
\end{equation}
Using the distribution of $\xns$ it can also be shown that
\[
\sup_{N\in\N}\E\bigl[\fxn_*\bigr]=\sup_{N\in\N}\frac{\E
[\xn_*]}{N}\leq3.
\]
An application of Markov's inequality then shows
that the sequence\break $\{\fxn_*\}_{N \in\N}$ is tight. Let $\overline x_*$ denote
a subsequential
weak limit of $\{\fxn_*\}_{N \in\N}$. Then\vspace*{1pt}
(\ref{mmnb}) clearly shows that almost surely, $\overline x_*=1$. Thus, as
$N \ra\infty$ $\fxn_*$ converges weakly to $1$ (see Theorem 1 of
\cite{HalWhi} for a more refined
calculation that also identifies the limit of the sequence of
stationary distributions centered
around $N$ and divided by $\sqrt{N}$).
We have shown that the sequence of stationary distributions
does not converge (even along a subsequence) to the value $2$, thus
demonstrating that the interchange of limits does not hold
even in this simple setting.

In addition, this example also demonstrates
that even in the presence of multiple invariant states,
the sequence of scaled stationary distributions
$(\fxns, \frenegns, \fmeasns)$, $N \in\N$,
could still converge to a limit. In the above example,
the explicit formula of the stationary
distribution of the $M/M/N$ queue was used to compute this limit,
which defeats the whole purpose of the approximation.
It is unclear whether, in the
presence of multiple invariant states, there is a general
methodology that does not rely on {a priori} knowledge of the
stationary distributions of the $N$-server queues, but
that would nevertheless allow one to identify when a limit exists
and, in that case,
identify which invariant state corresponds to the limit.

A minor modification of the above example shows that the interchange of
limits can also fail to hold in the presence of abandonment.
For the same sequence of queues described above, suppose that customers
abandon the queue according to a nontrivial patience time distribution $G^r$
satisfying Assumption \ref{ass-h} and having support in $(3, \infty
)$. For
each $N\in\N$,
consider the marginal state process $(\xn,\measn,\renegn)$. Suppose
that there exists $(2, \fmeas_0, \freneg_0)
\in{\cal S}_0$ such that almost surely, as $N \ra\infty$,
%
%
\begin{equation}
\label{eg-convint2} \bigl(\fxn(0), \fmeasn_0, \frenegn_0\bigr) \ra(2, \fmeas_0,
\freneg_0).
\end{equation}
Given the assumption imposed on the patience time
distribution, Assumption 2 fails to hold
because in this case $B_1=[1,3]$. By the previous argument,
Assumptions \ref{as-mean}, \ref{ass-en} and \ref{ass-h} are
satisfied. Therefore, by the fluid limit result stated as Theorem 3.6
of \cite{kanram08b}
(see also the proof of Theorem \ref{thm-convstat} of the current
paper) it
follows that almost surely, as $N \ra\infty$, $(\fxn, \fmeasn,
\frenegn)$
converges weakly to the unique solution $(\fx, \fmeas, \freneg)$ of
the fluid equations associated with $({\mathbf1}, 2, \fmeas_0,
\freneg_0)$. By the exponentiality of the service time distribution,
we have $\fx(t) = \fx(0)=2$ and $\fr(t)=0$ for each $t\geq0$.
On the other hand, let $\fyns=(\overline{\alpha}{}^{(N)}_{E,*},\fxns
,\fmeasns,\frenegns)$
be the stationary distribution associated with the fluid-scaled
state process, which exists by Theorem \ref{thm-SD}. By a simple coupling
argument, it can be shown that $\xn$ is stochastically dominated by the
corresponding
state $\tilde
X^{(N)}$ of an $M/M/N$ queue without abandonment that has the same arrival
process $\en$ and the same initial condition [i.e., $\PP(\tilde
X^{(N)}\geq
c)\geq\PP(\xn\geq c)$ for every $c>0$]. Together with the previous
discussion of the
case without abandonment, this can be used to show that $\{\fxn_*\}_{N
\in
\N}$ is tight and, for any $\ve> 0$,
$\limsup_{N\ra\infty}\PP(\fxns\geq(1+\ve)) = 0$, from which one can
conclude that $\fxns\rightarrow1$.
Thus, in this case too,
\[
\lim_{N \ra\infty} \lim_{t \ra\infty}
\fxn(t) = \lim_{N \ra\infty} \fxn_* = 1 \neq
2 = \lim_{t \ra\infty} \fx(t) = \lim_{t \ra\infty} \lim_{N \ra
\infty}
\fxn(t),
\]
where the limits are all in the sense of weak convergence.

\begin{appendix}\label{secPHR}

\section*{\texorpdfstring{Appendix: Proof of Theorem \lowercase{\protect\ref{thm-ergodic}}}%
{Appendix: Proof of Theorem 7.1}}

By Theorem 6.1 of \cite{MT}, to show that the Feller
process $\{Y_t,{\cal F}_t, t\geq0\}$
is ergodic, it suffices to show that the skeleton chain
$\{Y_n\}_{n \in\N}$ is $\psi$-irreducible and that
$\{Y_t,{\cal F}_t, t\geq0\}$
is positive Harris recurrent. This is done in Lemma~\ref{lemYirreducible}
and Theorem \ref{thmPHR}
below.
Let $\varrho$ be the quantity defined in condition (2) of Assumption
\ref{assPosF}, and define
\[
{\cal Z} \doteq \{(\alpha,0,\zerof,\zerof)\in {\cal Y}\dvtx
\alpha \in [\varrho + 1,\infty)\}.
\]
For each Borel subset $A$ of ${\cal Z}$,
let $\Gamma_A \subset [1+\varrho, \infty)$ be the Borel subset
obtained by projecting
${\cal Z}$ to its first coordinate:
%
%
\setcounter{equation}{0}
\begin{equation}\label{setA}
\Gamma_A \doteq  \{ \alpha \in \R\dvtx
(\alpha, 0, \zerof,\zerof) \in A\}.
\end{equation}
\setcounter{theorem}{0}
\begin{lemma} \label{lemPHR1} There exists a strictly positive
continuous function $C$ on~${\cal Y}$ such that for every $y=(\alpha,
x, \sum_{i=1}^k \delta_{u_i}, \sum_{j=1}^l \delta_{z_j})\in{\cal
Y}$, every Borel subset $A\subset{\cal Z}$ and every $t>2\varrho+1$,
\begin{eqnarray}
&&\PP_y\bigl(Y(t)\in A\bigr) \nonumber\\[-8pt]\\[-8pt]
&&\qquad\geq C(y)\int_{\alpha+2\varrho+1}^{\alpha+t}
\indd_{\Gamma_A}(\alpha+t-s)\bigl(1-F(\alpha+t-s)\bigr)\,
dF(s).\nonumber
\end{eqnarray}
\end{lemma}
\begin{pf}
At time $t$, if the state $Y(t)$ is in the set $A\subset{\cal Z}$,
this means that, by time~$t$, all customers in service at time $0$ with
residual service times $\{u_i, 1\leq i\leq k\}$, all customers in
queue at
time $0$ with residual patience times $\{z_j, 1\leq j\leq l\}$ and
those new
customers that arrived in the interval $[0,t)$ have completed service
(if they
entered service before time $t$) and have run out of their patience
(irrespective of whether or not they entered service).
Now, we consider a subset of $\{\omega\dvtx Y(t,\omega)\in A\}$, in which
(a) by time $2\varrho+1 < t$, all the initial customers with residual
patience times $\{z_j, 1\leq j\leq l\}$ and residual service times
$\{u_i, 1\leq i\leq k\}$ have finished service (if they entered
service) and run out of their patience (irrespective of whether or nor
they entered service), (b) the first new customer arrived after
$2\varrho+1$, finished service before $t$ and ran out of his/her
patience time before $t$, (c) the difference between $t$ and the
arrival time of that customer lies in~$\Gamma_A$, and (d) the second
new customer arrived after time $t$.  Let $\calQ_a$, $\calQ_{ad}$
and~$\calQ_{bd}$, respectively, be the events that property (a) holds,
properties (a)--(d) hold and properties~(b)--(d) hold.  Then, for $y\in
{\cal Y}$,
\[
\PP_y\bigl(Y(t)\in A\bigr) \geq \PP_y(\calQ_{ad}) = \PP_y(\calQ_{a})
\PP_y(\calQ_{bd}|\calQ_{a}),
\]
and, due to the independence assumptions on the service, patience and
interarrival distributions,  $\PP_y(\calQ_{bd}|\calQ_{a})$ is greater
than or equal to
\begin{eqnarray*}
&&
\int_{\alpha+2\varrho+1}^{\alpha+t}
G^r(\alpha+t-s)G^s(\alpha+t-s)\\
&&\qquad\quad\hspace*{3pt}{}\times\indd_{\Gamma_A}(\alpha+t-s)\bigl(1-F(\alpha+t-s)\bigr)\,
\frac{dF(s)}{1-F(\alpha)} \\
&&\qquad\geq
\frac{G^r(\varrho+1)G^s(\varrho+1)}{1-F(\alpha)}\int_{\alpha+2\varrho+1}^{\alpha+t}
\indd_{\Gamma_A}(\alpha+t-s)\bigl(1-F(\alpha+t-s)\bigr) \,dF(s),
\end{eqnarray*}
where the last inequality holds because $\alpha+t-s\geq \varrho+1$ when
$\alpha+t-s\in \Gamma_A$. Let $C(y) \doteq (\PP_y( \calQ_a)
G^r(\varrho+1)G^s(\varrho+1))/(1-F(\alpha))$.
Since, due to Assumption~\ref{assPosF}(2),
$G^r(A)>0$ and $G^s(A)>0$ for any interval
$A$ with length bigger than~$\varrho$,
$\PP_y(\calQ_a)$, as a function of $y\in \cal
Y$, is strictly positive and continuous.
Thus $C$ is a strictly positive and continuous function
on $\cal Y$, and the lemma is proved.
\end{pf}
\begin{defn}
Any Markov process $\{X_t\}$ with topological state spa\-ce~${\cal X}$ is said
to be $\psi$-irreducible if and only if there exists a $\sigma
$-finite measure
$\psi$ on~${\cal B}({\cal X})$, the Borel $\sigma$-algebra on ${\cal
X}$ such
that for every $x\in{\cal X}$ and $B \in{\cal B}({\cal X})$,
\[
\int_0^\infty\PP_x\bigl(X(t)\in B\bigr)\,dt >0 \qquad\mbox{if } \psi(B)>0.
\]
\end{defn}

Let $\psi=m\times\delta_0\times\delta_0\times\delta_0$, where
$m(A) =
\overline{m}(A \cap[\varrho+1, \infty))$, where $\overline{m}$ is Lebesgue
measure. Clearly, $\psi$ is a $\sigma$-finite measure on $({\cal
Y},{\cal B}({\cal Y}))$.
\begin{lemma}\label{lemYirreducible}
The Markov process $\{Y_t, {\cal F}_t\}$ is $\psi$-irreducible and
the discre\-te-time Markov chain $\{Y(n)\}_{n\in\N}$ is $\psi$-irreducible.
\end{lemma}
\begin{pf}
Let $B\in{\cal B}({\cal Y})$ be such that $\psi(B)>0$. Then
$\psi(B\cap{\cal Z})>0$ by the
definition of $\psi$.
Let $\Gamma_{B\cap {\cal Z}}$ be the set defined in (\ref{setA}) with
$A = B \cap {\cal Z}$
and suppose $m(\Gamma_{B\cap{\cal Z}})>0$. Fix $y\in\cal Y$. It follows
from Lemma \ref{lemPHR1}
that there exists a strictly positive function $C$ on $\cal Y$
such that
\begin{eqnarray*}
&&\int_0^\infty \PP_y\bigl(Y(t)\in B\cap {\cal Z}\bigr)\,dt \\[-2pt]
&&\qquad \geq
\int_{2\varrho+1}^\infty \PP_y\bigl(Y(t)\in B\cap {\cal Z}\bigr)\,dt \\[-2pt]
&&\qquad \geq
\int_{2\varrho+1}^\infty C(y)\biggl( \int_{\alpha+2\varrho+1}^{\alpha+t}
\indd_{\Gamma_{B\cap {\cal Z}}}(\alpha+t-s)\bigl(1-F(\alpha+t-s)\bigr)
\, dF(s) \biggr) \, dt \\[-2pt]
&&\qquad = C(y) \bigl(1-F(\alpha+2\varrho+1)\bigr) \int_{\Gamma_{B\cap {\cal
Z}}}\bigl(1-F(t)\bigr)\, dt\\[-2pt]
&&\qquad > 0,
\end{eqnarray*}
where the equality follows from Fubini's theorem and the last inequality
holds because $C(y)>0$, $m(\Gamma_{B\cap{\cal Z}})>0$ and
$1-F(x)>0$ for every $x\in[0,\infty)$ by Assumption
\ref{assPosF}(3). This establishes the first assertion. On the other
hand, for $n> 2\varrho+1$,
\[
\PP_y\bigl(Y(n)\in B\bigr) \geq C(y)\int_{\alpha+2\varrho+1}^{\alpha+n}
\indd_{\Gamma_{B\cap {\cal Z}}}(\alpha+n-s)\bigl(1-F(\alpha+n-s)\bigr)\, dF(s).
\]
By Assumption \ref{assPosF}(3) and the fact that $m(\Gamma_{B\cap
{\cal Z}})>0$, it follows that\break $\PP_y(Y(n)\in B)>0$ for all $n$
sufficiently large. Hence, $\{Y(n)\}_{n\in\N}$ is
$\psi$-irreducib\-le.\vadjust{\goodbreak}~%
\end{pf}

%

For each $y\in{\cal Y}$, $B\in{\cal B}({\cal Y})$ and each probability
measure $\Pi$ on $[0,\infty)$, let
\[
{\cal K}_\Pi(y,B)=\int_0^\infty\PP_y\bigl(Y(t)\in B\bigr) \Pi(dt).
\]

%
\begin{lemma} \label{lem-T} There exists a probability measure $\Pi$
on $[0,\infty)$ and a function $T\dvtx{\cal Y}\times{\cal B}({\cal Y})
\ra\R_+$ such that:
\begin{longlist}[(1)]
\item[(1)] ${\cal K}_{\Pi}(y,B)\geq T(y,B)$ for all $y\in{\cal Y}$ and
every Borel measurable set $B\in{\cal B}({\cal Y})$;
\item[(2)] $T(y,{\cal Y})>0$ for all $y\in{\cal Y}$;
\item[(3)] $T(\cdot,B)$ is lower-semicontinuous for every $B\in{\cal
B}({\cal Y})$.
\end{longlist}
\end{lemma}
\begin{pf}
Let $C$ be the strictly positive, continuous function $C$ of Lem\-ma~\ref{lemPHR1}.
Let $\Pi$ be a probability measure with density function $e^{-(t-2\varrho-1)}$
on $[2\varrho+1,\infty)$. For each $y\in\cal Y$ and $B\subset{\cal Z}$,
define
\[
T(y,B)\doteq
C(y)e^{\alpha+2\varrho+1}\int_{\alpha+2\varrho+1}^\infty e^{-s}\,dF(s)
\int_0^\infty \bigl(1-F(t)\bigr)\indd_{\Gamma_B}(t)e^{-t}\,dt,
\]
and
$T(y,{\cal Y}\setminus{\cal Z})=0$. It is easy to see that for any Borel
measurable set $B \in{\cal B}(\cal Y)$, $T(y,B) = T(y,B \cap\cal Z)$ and
$T(\cdot, B)$ is continuous. Moreover, $T(y,\cal Y)=T(y,\cal Z)>0$.
Now, fix $y\in\cal Y$ and $B\in{\cal B}({\cal Y})$. By Lemma
\ref{lemPHR1}, we have
\begin{eqnarray*}
&&{\cal K}_\Pi(y,B) \\
&&\qquad= \int_0^\infty\PP_y\bigl(Y(t)\in
B\bigr)e^{-(t-2\varrho-1)} \,dt \\
&&\qquad\geq\int_{2\varrho+1}^\infty
\PP_y\bigl(Y(t)\in B\cap{\cal Z}\bigr)e^{-(t-2\varrho-1)} \,dt \\
&&\qquad \geq
\int_{2\varrho+1}^\infty C(y)
\biggl(\int_{\alpha+2\varrho+1}^{\alpha+t} \indd_{\Gamma_{B\cap {\cal
Z}}}(\alpha+t-s)\bigl(1-F(\alpha+t-s)\bigr) \,dF(s) \biggr)\\
&&\qquad\quad\hspace*{23.4pt}{}\times e^{-(t-2\varrho-1)}\,dt  \\
&&\qquad = C(y)e^{\alpha+2\varrho+1}
\int_{\alpha+2\varrho+1}^\infty e^{-s}\,dF(s)
\int_0^\infty \bigl(1-F(t)\bigr)\indd_{\Gamma_{B\cap {\cal Z}}}(t)e^{-t}\,dt\\
&&\qquad= T(y,B\cap{\cal Z})=T(y,B).
\end{eqnarray*}
Thus we have proved the lemma.
\end{pf}
\begin{theorem} \label{thmPHR}
The Markov process $Y$ is positive Harris recurrent.
\end{theorem}
\begin{pf}
Lemma \ref{lem-T} shows that $Y$ is a so-called $T$ process
(cf. Section~3.2 of \cite{MT}) and Lemma \ref{lemYirreducible}
shows that $Y$ is $\psi$-irreducible. Now, Theorem 3.2 of \cite{MT} states
that any $\psi$-irreducible $T$ process $Y$ is
positive Harris recurrent if $Y$ is bounded in probability on average, that
is, for each $y\in\cal Y$ and $\varepsilon>0$, there exists a compact set
$B\in{\cal B}({\cal Y})$ such that
\[
\liminf_{t\ra\infty}\frac{1}{t}\int_0^t
\PP_y\bigl(Y(s)\in B\bigr) \,ds \geq1-\varepsilon.
\]
However, this is satisfied
by the state process $Y$ due to
Lemma \ref{lemLtight}. So we have the desired result.
\end{pf}
\end{appendix}

\section*{Acknowledgment}

We would like to thank Haya Kaspi for observing that the state process
would not be Feller if $\alpha_E$ were chosen to be the forward, rather
than the backward, recurrence time.


%

%
\printaddresses

\end{document}